    \documentclass{amsart}
\usepackage{amsmath,amsthm}
\usepackage{verbatim}
\usepackage{marvosym}
\usepackage{graphicx}
\usepackage{color}
\usepackage{epsfig}
\usepackage{rotating}
\usepackage{lscape}
\usepackage[all]{xy}
\usepackage{mathrsfs, euscript, mathdesign}

\theoremstyle{plain}
\newtheorem{thm}{Theorem}[section]
\newtheorem{cor}[thm]{Corollary}

\newtheorem{prop}[thm]{Proposition}
\newtheorem{lemma}[thm]{Lemma}

\theoremstyle{definition}
\newtheorem{definition}[thm]{Definition}
\newtheorem{remark}[thm]{Remark}

\makeatletter
\newtheorem*{rep@theorem}{\rep@title}
\newcommand{\newreptheorem}[2]{%
\newenvironment{rep#1}[1]{%
 \def\rep@title{#2 \ref{##1}}%
 \begin{rep@theorem}}%
 {\end{rep@theorem}}}
\makeatother

\newtheorem{theorem}{Theorem}
\newreptheorem{theorem}{Theorem}

\newcommand{\kate}[1]{{\color[rgb]{.8,.3,.1}{#1}}}

\newcommand{\Q}{\ensuremath{\mathbb{Q}}}
\newcommand{\R}{\ensuremath{\mathbb{R}}}
\newcommand{\Z}{\ensuremath{\mathbb{Z}}}
\newcommand{\C}{\ensuremath{\mathbb{C}}}
\newcommand{\PP}{\ensuremath{\mathbb{P}}}
\renewcommand{\H}{\ensuremath{\mathbb{H}}}

\newcommand{\SL}{\text{SL}}
\newcommand{\PSL}{\text{PSL}}
\newcommand{\GL}{\text{GL}}
\newcommand{\tr}{\text{tr}}

\newcounter{nootje}
\renewcommand\check[1]
  {\marginpar{\tiny\begin{minipage}{20mm}\begin{flushleft}\thenootje :
  #1\end{flushleft}\end{minipage}}\addtocounter{nootje}{1}}

\newcommand{\mat}[4]{\left(\begin{array}{cc}#1 & #2 \\ #3 & #4 \end{array}\right)}

\begin{document}
\title[Gonality]{Gonality and genus of canonical components
of character varieties}

\author[K.L. Petersen, A.W. Reid ]{ Kathleen L. Petersen \\
 Alan W. Reid}
\maketitle


\section{Introduction}

Throughout the paper, $M$ will always 
denote a complete, orientable finite volume hyperbolic 3-manifold
with cusps.  By abuse of notation we will denote by $\partial M$ 
to be the boundary of the compact manifold obtained from $M$ by truncating the
cusps.

Given such a manifold, the $\SL_2(\C)$ character variety of $M$, $X(M)$, is a
complex algebraic set associated to representations of $\pi_1(M)
\rightarrow \SL_2(\C)$ (see \S 4 for more details).  Work of Thurston showed 
that
any irreducible component of such a variety containing the character
of a discrete faithful representation has complex dimension equal to
the number of cusps of $M$.  Such components are called {\em canonical
  components} and are denoted $X_0(M)$.  
Character varieties have been fundamental
tools in studying the topology of $M$ (we refer the reader to
\cite{MR1886685} for more), and 
canonical components carry a
wealth of topological information about $M$, including containing 
subvarieties associated to Dehn fillings of $M$.  

When $M$ has exactly
one cusp, any canonical component is a complex curve.  The aim of this
paper is to study how some of the natural invariants of these complex
curves correspond to the underlying manifold $M$.  In particular, we
concentrate on how the {\em gonality} of these curves behaves in
families of Dehn fillings on 2-cusped hyperbolic manifolds.
More precisely, we study families of 1-cusped
3-manifolds which are obtained by Dehn filling of
a single cusp of a fixed 2-cusped hyperbolic 3-manifold, $M$. We
write $M(-,r)$ to denote the manifold obtained by $r=p/q$ filling of
the second cusp of $M$.  

To state our results we introduce the
following notation.
If $X$ is a complex curve, we write $\gamma(X)$ to denote the gonality
of $X$, $g(X)$ to be the (geometric) genus of $X$ 
and $d(X)$ to be the degree (of the specified embedding) of $X$. The gonality of a curve is the lowest degree of a map from that curve to $\C$.  Unlike genus, gonality is not a topological invariant of curves, but rather is intimately connected to the geometry of the curve.  There are connections between gonality and genus, most notably the Brill-Noether theorem which gives an upper bound for gonality in terms of genus (see \S~\ref{section:genusanddegree}) but in some sense these are orthogonal invariants. For example, all hyperelliptic curves all have gonality two, but can have arbitrarily high genus.  Moreover,  for $g>2$, there are curves of genus $g$ of different gonality.  We refer the reader to \S~\ref{section:DefinitionsandAlgebraicGeometry} for precise definitions.

Our first theorem is the following.

\begin{thm}\label{theorem:1}
Let $M$ be a finite volume hyperbolic 3-manifold with two cusps. If $M(-,r)$ is hyperbolic, then there is a positive constant $c$ depending only on $M$ such that $$\gamma(X_0(M(-,r)))\leq c.$$
\end{thm}

The figure-8 knot complement and its so-called sister manifold are
both integral surgery of one component of the Whitehead link
complement.  Their canonical components are well-known to be an
elliptic curve, and a rational curve, respectively. These have
gonality 2 and 1.  Therefore, Theorem~\ref{theorem:1} is best
possible.

Our techniques can also be used to obtain information about the genera and degree of related varieties as we now discuss.  
The inclusion map from $\pi_1(\partial M)$ to $\pi_1(M)$ induces a map from $X(M)$ to the character variety of $\partial M$.  We let $A(M)$ denote the image of this map and let $A_0(M)$ denote the image of a canonical component. 
 When $M$ has two cusps, the variety $A(M)$ naturally sits in $\C^4(m_1,l_1,m_2,l_2)$ where the $m_i$ and $l_i$ are a choice of framing (we often refer to these
as meridional and longitudinal parameters for the $i^{th}$ cusp). For $r=p/q\neq 0$ in lowest terms, we define the naive height of $r$ to be $h(r) = \max\{|p|,|q|\}$.  
 To avoid cumbersome notation, we  define $h(0)=h(\infty)=1$.  
\begin{thm}\label{theorem:genus}
Let $M$ be a finite volume hyperbolic 3-manifold with two cusps.  If $M(-,r)$ is hyperbolic, then
there is a  positive  constant $c$  depending only on $M$ and the framing of the second cusp such that 
\[   g(A_0(M(-,r))) \leq c \cdot h(r)^2.\]
\end{thm}


Using work of Dunfield \cite{MR1695208} (see also \S 4) we deduce
\begin{cor} Let $M$ be a finite volume hyperbolic 3-manifold with two cusps. If $M(-,r)$ is hyperbolic, 
 and $|H^1(M(-,r); \Z/2\Z)|=2$ then 
 there is a  positive  constant $c$  depending only on $M$ and the framing of the second cusp such that 
\[ g(X_0(M(-,r))) \leq c \cdot h(r)^2. \]
\end{cor}

Unlike the genus or gonality of a curve, the degree of a variety is inherently tied to its embedding in ambient space.   We obtain  upper bounds for the degree of a `natural model' of $A_0(M(-,r))$. 

The proof of Theorem~\ref{theorem:1} breaks
naturally into cases, depending on whether or not one (or both) 
cusps of $M$ are geometrically isolated from the other cusp (see \cite{MR1202390} and \S~\ref{section:Isolationofcusps}).  Geometric isolation also plays a role in our determination of degree bounds. 

\begin{thm}\label{theorem:degree}
Let $M$ be a finite volume hyperbolic 3-manifold with two cusps. If $M(-,r)$ is hyperbolic, then  there is a  positive  constant $c_1$ depending only on $M$ and the framing of the second cusp such that
\[ d(A_0(M(-,r))) \leq  c_1\cdot h(r).\] 
If one cusp is geometrically isolated from the other cusp, then there is a  positive  constant $c_2$ depending only on $M$ and the framing of the second cusp, such that 
\[ d(A_0(M(-,r))) \leq c_2.\] 
\end{thm} 
Work of Hoste and Shanahan
\cite{MR2047468} demonstrates that the degree
bounds in Theorem~\ref{theorem:degree}
are  sharp.

In contrast with the main results of this paper, we also show
that by filling two cusps of a three
cusped manifold, we can construct examples of knots 
(so-called double twist knots) for which the gonality
of the canonical component can be made arbitrarily large (see \S 11.2 for
details).
Existence of families of curves with large gonality has been the
subject of much recent research, including connections between the existence of towers of  curves whose gonality goes to infinity and 
expander graphs \cite{MR2922374}.

We also remark
on the gonalities of the character varieties of the  once-punctured torus bundles of tunnel number
one \cite{OPTBTNO}.  These manifolds are all integral surgeries on one cusp of the
Whitehead link and the corresponding  gonalities are all  equal to one or two.  
Both the double twist knot examples and the once-punctured torus bundles of tunnel number
one examples illustrate linear growth in the genus of the character varieties.  We do not know if a
quadratic bound such as in Theorem~\ref{theorem:genus} is optimal.

\section{Acknowledgements}

The first author would like to thank the University of Texas at Austin and the Max Planck Institut f\"ur Mathematik for their hospitality while working on this manuscript.  She would also like to thank Martin Hils for helpful discussions. This work was partially supported by a grant from the Simons Foundation (\#209226 to Kathleen Petersen), and by the National Science Foundation (to Alan Reid).

\section{Algebraic Geometry and Preliminary  Lemmas}\label{section:DefinitionsandAlgebraicGeometry}

\subsection{Preliminaries}\label{section:preliminaries}

We begin with some algebraic geometry. We follow Shafarevich \cite{MR1328833} and Hartshorne \cite{MR0463157}.  The `natural models' of the character varieties of interest are complex affine algebraic sets, which 
typically have singularities at infinity.   We will  rely on an analysis of  rational maps between possibly singular affine sets. Such a map $f:X\dashrightarrow Y$  behaves (on most of $X$ and $f(Y)$) as  a combination of a very nice map and a map with controlled bad parts.  We now present some algebraic geometry to make this precise.

Let $\C[T]$ be the polynomial ring with coefficients in $\C$ in the variables $T_1, \dots, T_n$.  A (Zariski) closed subset of $\C^n$ is a subset $X\subset \C^n$ consisting of the common vanishing set of a finite number of polynomials in $\C[T]$.  We let $\mathfrak{U}_X$ denote the {\em ideal of $X$}, the ideal in $\C[T]$ consisting of all polynomials which vanish on $X$.   If $I$ is an ideal in $\C[T]$ we let $V(I)$ denote the \em vanishing set \rm of $I$ in $\C^n$, and if $f\in \C[T]$ let $V(f)$ be the vanishing set of $f$ in $\C^n$.
The {\em coordinate ring} of $X$  is $\C[X]=\C[T]/{\mathfrak U}_X$.  This ring is the ring of all polynomials on $X$,  up to the equivalence of being equal on $X$; adding any polynomial in ${\mathfrak U}_X$ to a polynomial $f$ does not change the value of $f$ on $X$.  If a closed set $X$ is irreducible then the field of fractions of the coordinate ring $\C[X]$ is the {\em function field} or {\em field of rational functions } of $X$; it is denoted $\C(X)$. 

 A function $f$ defined on $X$ with values in $\C$ is {\em regular} if there exists a polynomial $F(T)$ with coefficients in $\C$ such that $f(x)=F(x)$ for all $x\in X$. 
A map $f:X \rightarrow Y$ is {\em regular} if there exist $n$ regular functions $f_1, \dots, f_n$ on $X$ such that $f(x) = (f_1(x), \dots, f_n(x))$ for all $x\in X$. 
Any regular map $f:X \rightarrow Y$ defines a  pullback homomorphism of $\C$-algebras   $f^*:\C[Y] \rightarrow \C[X]$ as follows.   If $u:X\rightarrow Z$ we define $v:Y \rightarrow Z$ by $v(x)=u(f(x))$.  We define $f^*$ by $f^*(u)=v$. 
 A {\em rational map} $f:X\dashrightarrow Y \subset \C^m$ is an $m$-tuple of rational functions $f_1,\dots, f_m \in \C(X)$ such that, for all points $x\in X$ at which all the $f_i$ are defined, $f(x) = (f_1(x), \dots, f_m(x))\in Y.$  We define the pullback $f^*$ as above.
A rational map $f:X\dashrightarrow Y$ is called {\em dominant }  if  $f(X)$ is (Zariski) dense in $Y$. 
A rational map $f:X\dashrightarrow Y$ is called {\em birational} if $f$ has a rational inverse, and in this case 
we say that $X$ and $Y$ are birational.   A regular map $f:X\rightarrow Y$ is {\em finite} if $\C[X]$ is integral over $\C[Y]$. 
 These definitions naturally extend to the projective setting. A {\em morphism} $\varphi :X\rightarrow Y$ is a continuous map such that for every open $V\subset Y$ and every regular function $f:Y\rightarrow \C$ the function $f\circ \varphi: \varphi^{-1}(Y)\rightarrow \C$ is regular.


Character varieties are naturally affine sets.  It is useful to consider an embedding of such an affine set into projective space.
Given an affine variety $X$, a {\em projective completion} of $X$ is the Zariski closure of an embedding of $X$ in some projective space.  A  {\em smooth projective model} of $X$ is a  smooth projective variety birational to a projective completion of $X$.  Such a model is unique up to isomorphism if $X$ is a curve.    The {\em degree } of any projective curve $X$, denoted $d(X)$ is the maximum number of points of intersection of $X$ with a hyperplane not containing any component of $X$.  This depends on the specific embedding of $X$.

The gonality of a curve is, in the most basic sense, the degree of a map from that curve to $\C^1$.  Intuitively the degree of a map  is  the number of preimages of  `most points' in the image.
Let $X$ and $Y$ be irreducible varieties of the same dimension and $f:X\dashrightarrow Y$ a dominant rational map.  The degree of the field extension $f^*(\C(Y))\subset \C(X)$ is finite and   is called the {\em degree} of $f$. That is, 
\[ \text{deg} f = [\C(X):f^*(\C(Y))].\] 
This is well defined for rational and regular maps.

We are now in a position to see that dominant rational maps between possibly singular curves can be decomposed in a pleasing manner. 

\begin{prop}\label{prop:mapsarenice} 
Let $X$ be a possibly singular irreducible affine variety of dimension $n$ and $Y$ an algebraic set of dimension $n$. 
Assume that $g:X\dashrightarrow Y$ is a dense rational map.  
\begin{enumerate}
\item $\deg(g)=d$ is finite.  
\item For all but finitely many co-dimension one subvarieties $W$ of $X$, \[ \deg(g\mid_W)\leq d.\]
\item\label{item:3} If $n\leq 2$ then $g^{-1}(y)$ consists of at most $d$ points except for perhaps finitely many $y \in Y$.  In $X$ there are finitely many irreducible sets of co-dimension one such that if $x$ is not contained in the union of these sets  then $g^{-1}(g(x))$ contains only  finitely many points.
\end{enumerate}

\end{prop}

\begin{proof}
By Hironaka \cite{MR0199184} one can resolve the singularities of a variety by a proper birational map (which has degree one).   Therefore, the variety $X$ is birational to a smooth projective variety $X'$, and similarly $Y$ is birational to a smooth projective $Y'$.   The rational map $g:X \dashrightarrow Y$ induces a rational map $g':X' \dashrightarrow Y'$ of the same degree by pre and post composition with the birational inverses,  
\[ X \overset{\alpha}{\dashrightarrow} X' \overset{g'}{\rightarrow} Y' \overset{\beta}{\dashrightarrow}  Y.\]
By Hironaka's resolution of indeterminacies \cite{MR0199184}, $g'$ can be resolved into a regular map by a sequence of blowups.  
Therefore we can take $g'$ to be  a proper morphism. 
By Stein Factorization (see \cite{MR0463157} Corollary 11.5), a proper morphism factors as a proper morphism with connected fibers followed by a finite map. It follows that the degree is finite.
By the Weak Factorization Theorem \cite{MR2013783} these  birational maps (since they are over $\C$) can be decomposed as a series of blow-ups and blow-downs over smooth complete varieties. 

A blow up or  blow down is defined for all points outside of a codimension one set.  Therefore, outside of a finite union of codimension one sub varieties, the map $g$ is the composition of  isomorphisms, a proper morphism with connected fibers, and a finite map. If $W\subset X$ is a codimension one subvariety not in this set, then the degree of $g\mid_W$ is at most the degree of $g$.

If $X$ is a curve, then  part~\ref{item:3} follows from the observation that in this case, Stein Factorization implies that $g'$ is a finite map.  If $X$ is a surface, then  it suffices to show content of the lemma for $X$ and $Y$  smooth projective surfaces, and $g:X\rightarrow Y$ a dense morphism.   By Stein factorization such a $g$ can be decomposed into  a map with connected fibers followed by  a finite map.  For the finite map, the  number of pre-images of a point is bounded above by the degree.  Since the maps are between surfaces, the map with connected fibers must be birational, so part~\ref{item:3} follows. 



\end{proof}


 \begin{lemma}\label{lemma:irreducible} Let $f:X\dashrightarrow Y$ be a dominant rational map of affine algebraic sets.  If $X$ is irreducible, then $Y$ is irreducible.  (That is, the closure of $f(X)$ is irreducible.)
\end{lemma}

\begin{proof} By Proposition~\ref{prop:mapsarenice}  it suffices to assume that $f$ is regular. 
The map $f$ induces a map $f^*:\C[Y] \rightarrow \C[X]$ of coordinate rings. If $X\subset \C[T]$ and $Y\subset \C[S]$ then $\C[X]=\C[T]/{\mathfrak U}_X $ and $\C[Y]=\C[S]/{\mathfrak U}_Y$.   Since $f(X)$ is dense in $Y$, it follows that $f^*$ is an isomorphic inclusion (see \cite{MR1328833} page 31).  An algebraic set is irreducible if and only if the vanishing ideal is prime.  Therefore ${\mathfrak U}_X $ is prime and $\C[X]$ has no zero divisors.  We conclude that $\C[Y]$ has no zero divisors, and therefore ${\mathfrak U}_Y$ is prime and $Y$ is irreducible. 

\end{proof}

 \subsection{Gonality}
 
The algebraic definition of gonality is as follows. 
 \begin{definition}
The {\em gonality} of a curve $X$ in $\PP^n$ is 
\[ \gamma(X) =  \min\{ [\C(X):\C(x)]:x\in \C(X)-\C\}.  \]
\end{definition}
Here, $x \in \C(X)$ is generated by a single element. Since  $\C(X)$ is the function field of $X$ and consists of the rational functions on $X$, we can reformulate the definition; the gonality of $X$ is the minimal degree of a dominant rational map from $X$ to $\PP^1$. Therefore, gonality is a birational invariant; two varieties which are birational (but not necessarily  isomorphic) have equal gonality.  If $X$ is an affine curve, $X$ is birational to a smooth projective model and as such the gonality of $X$ and $X'$ are equal, and this is the minimal degree of a dominant rational map from $X$ to $\C$.

By Noether's Normalization theorem, if $X$ is complex projective variety of dimension $n$, then there is a linear subspace disjoint from $X$ such that the projection  onto that subspace is finite to one and surjective.   This is also true for affine varieties; if $Z$ is an  affine variety of dimension $n$ there is a finite degree map from $Z$ onto a dense subset of  $\C^n$.  
Therefore, gonality is always finite.

\begin{remark}Gonality of an affine variety is sometimes defined as the minimal degree of a regular map to $\C$ (or of a finite map to $\C$).  This (regular) gonality is easily seen to be greater than or equal to the (rational) gonality.  In fact, there are examples where these two notions differ.  The affine curve determined by the equation $xy=1$ is not isomorphic to $\C$ and so the (regular) gonality is not one.  One can see that there is a regular degree two map onto $\C$, by projecting onto the line $y=-x$.  Therefore, the (regular) gonality is two. Projection onto the $x$ coordinate is a dominant map to $\C$ as the image is $\C-\{0\}$ and has degree one; the (rational) gonality is one.

In the projective setting the gonality defined by using rational, regular, or finite maps are all equal, and are therefore equal to the (rational) gonality of a corresponding affine variety.  If there is a dominant rational map $g:X \dashrightarrow Y $ between affine (possibly singular) curves then this induces a finite dominant morphism $f$ between  smooth projective models $X'$ and $Y'$ for $X$ and $Y$, where the degree of $g$ and the degree of $f$ are the same.   

Gonality can also be defined using the language of Riemann surfaces instead of complex curves.  Here the definitions involving rational  or regular maps correspond to meromorphic or holomorphic functions from the Riemann surface to $\PP^1$.  Similarly, one can define the gonality in terms of the minimal degree of a branched or unbranched covering of $\PP^1$ by the Riemann surface.

 \end{remark}

\subsection{Key Lemmas}

 The following lemma will be essential for our gonality bounds.

\begin{lemma}\label{lemma:abstractgonalitybounds}
Let $g:X\dashrightarrow Y$ be a degree $d$ dominant rational map  of irreducible (possibly singular) affine or projective curves. Then 
\[   \gamma (Y) \leq    \gamma (X) \leq  d \cdot   \gamma (Y).    \]
\end{lemma}

\begin{proof}
The second inequality is clear.  To show the first, we follow \cite{MR2335995} Proposition A.1, including the proof for completeness. 
Let $f\in \C(X)$ be a degree $d$ map realizing the gonality of $X$.
Let $p(T)\in \C(Y)[T]$ be the   characteristic polynomial of $f$ when $f$ is viewed as an element in the extension $\C(X)$ of $\C(Y)$.  Let $M$ be a finite extension of $\C(Y)$ such that $p(T)$ factors into the product of  $r$ monomials $(T-f_i)$, with $f_i\in M$. 

Viewed as a function in $M$, the degree of $f$ is $d [M:\C(X)]$.  This is also the degree of the $f_i$ since they are all in the same ${\rm Aut}(M/\C(Y))$ orbit.  The polar divisors of a coefficient of $p$ viewed in $M$ is at most the sum of the polar divisors of the $f_i$.  (These polar divisors are just the effective divisors $(f_i)_{\infty}$.) 

Therefore, each coefficient has degree at most $d r [M:\C(X)]=d [M:\C(Y)]$ as a function in $M$, and therefore degree at most $d$ as a function in $\C(Y)$. Since $f$ is non-constant, at least one of these coefficients is non-constant.  Therefore the gonality of $Y$ is at most $d$.

\end{proof}

We now introduce some notation.  Commonly $\alpha(x) \in \Theta(\beta(x))$ is written  to mean that there are positive constants $c_1$ and $c_2$ independent of $x$ such that 
\[ c_1\beta(x) \leq \alpha(x) \leq c_2 \beta(x).\]  Notice that if $\alpha(x) \in \Theta(\beta(x))$ then also 
\[ c_2^{-1}\alpha(x) \leq \beta(x) \leq c_1^{-1}\alpha(x)\] so $\beta(x) \in \Theta(\alpha(x))$.  This defines an equivalence relation. 
\begin{definition}\label{definition:sim}
For positive quantities $\alpha=\alpha(x)$ and $\beta=\beta(x)$ we write $\alpha\sim \beta$ if $\alpha \in \Theta(\beta)$.

We will use $\sim$ in conjunction with quantities such as  $\gamma(X_0(M(-,r)))$ to mean  that the constants do not depend on $r$ but only on $M$ and the framing of $\partial M$. Note that $\alpha(x)\sim 1$ means that there is a positive constant $c$ independent of $x$ such that $\alpha(x)\leq c$. 
\end{definition}

We now compute the gonality of the curves given by  $x^p-y^q$ and $x^py^q-1$ in $\C^2(x,y)$.  First we establish some notation for these sets.

\begin{definition}\label{definition:psi} Let $p$ and $q$ be relatively prime integers with $q>0$. 
With $r=p/q$, let 
\[ \mathcal{C}(r)= V(x^{|p|}y^q-1) \cup V(x^{|p|}-y^q)\]
in $\C^2(x,y)$ and let 
\[ \mathcal{H}(r) = V(x^{|p|}y^q-1) \cup V(x^{|p|}-y^q) \]
in $\C^4$ where $x$ and $y$ are among the coordinates. 
(Therefore, $\mathcal{H}(r)$ and $\mathcal{C}(r)$ are each the union of  two hypersurfaces unless $p=0$.)

\end{definition}

\begin{lemma}\label{lemma:keygonality}
Assume that $p$ and $q$ are relatively prime integers. Then  the gonality of each irreducible component of $\mathcal{C}(r)\subset \C^2(x,y)$  is one.
\end{lemma}

\begin{proof}
  If either $p$ or $q$ is zero, the polynomial is $x=1$ or $y=1$ and
  is therefore naturally a $\C^1$ and has gonality one.  Now assume
  that $p,q> 0$.  We will compute the gonality of $V(x^py^q-1)$ as the
  vanishing set of $x^p-y^q$ since this vanishing set is birationally
  equivalent to the vanishing set of $x^py^q-1$ under the map $(x,y)\mapsto
  (x,y^{-1})$.  

Since $p$ and $q$ are relatively prime, there are
  integers $a$ and $b$ such that $ap+bq=1$. Let $\varphi:\C^2(x,y)
  \dashrightarrow \C^2(x',y')$ be the rational map given by
\[ (x,y) \mapsto (x^by^{-a},x^{p}y^{q})=(x',y').\]
This has an inverse $\varphi^{-1}$ given by 
\[ (x',y') \mapsto ((x')^q(y')^{a},(x')^{-p}(y')^b). \]
The map  $\varphi$ is birational, so $\gamma(V(x^py^q-1))$ is equal to the gonality of the image of $V(x^py^q-1)$ under $\varphi$. This image is dense in $y'=1$.   As $y'=1$ naturally defines a $\C^1$ the gonality of $V(y'-1)$ and $V(x^py^q-1)$ are both one.

\end{proof}

\subsection{Gonality and Projection}
\

We will now prove a few lemmas which will be essential in  bounding the gonalities we are interested in. In $\S~\ref{section:gonalityindehnsurgeryspace}$ we will relate the situation below to that of the character varieties of interest. Recall the definition of $\mathcal{C}(r)$  and  $\mathcal{H}(r)$  from  Definition~\ref{definition:psi}.   By Lemma~\ref{lemma:keygonality} the gonality of (each irreducible component of) $\mathcal{C}(r) \subset \C^2$ is equal to one.

\begin{remark}
If a curve $C$ is reducible, by  $\gamma(C)$ we mean the maximal gonality of any irreducible curve component of $C$. 
\end{remark}

\begin{definition}\label{definition:varpiproj}
Define the maps $\varpi_1$ and $\varpi_2$ to be the   projection maps 
\[\varpi_i:\C^4(x_1,y_1,x_2,y_2)\rightarrow \C^2(x_i,y_i)\] for $i=1,2$.
 \end{definition}

By Definition~\ref{definition:psi}, with $\mathcal{H}(r) \subset$ $ \C^4(x_1,y_1,x_2,y_2)$ and $\mathcal{C}(r) \subset \C^2(x_2,y_2)$ then $\mathcal{C}(r)= \varpi_2(\mathcal{H}(r)).$

\begin{lemma}\label{lemma:V(r)curves}
Let $V$ be an irreducible surface in $\C^4(x_1,y_1,x_2,y_2)$ with $\dim_{\C}\varpi_i(V)>0$.  Then $V \cap \mathcal{H}(r)$ consists of a finite union of curves for all but finitely many $r$.   \end{lemma}

\begin{proof}
The dimension of $V$ is two, and the dimension of $\mathcal{H}(r)$ is three.  Therefore,  each irreducible component of their  intersection  has dimension at least one (see \cite{MR0463157} Section I.7). As a result, the dimension of any irreducible component of  $V \cap \mathcal{H}(r)$ is one, or two. 

Assume there are infinitely many $r$ such that $V \cap \mathcal{H}(r)$ contains an irreducible component of dimension two. 
The set  $\mathcal{H}(r)$ is of the form $\C^2 \times \mathcal{C}(r)$, and in $\C^2(x_2,y_2)$  any two distinct $\mathcal{C}(r)$ intersect in  only finitely many points. We conclude that 
  $V\subset \C^2\times P$ for some point $P$.    Therefore, $\varpi_2(V)=P$  has dimension zero, which does not occur by hypothesis.  It follows that $V \cap \mathcal{H}(r)$ consists of a finite union of curves for all but finitely many $r$.

\end{proof}

\begin{lemma}\label{lemma:key1} Let $V$ be an irreducible surface in $\C^4(x_1,y_1,x_2,y_2)$ with $\dim_{\C}\varpi_i(V)>0$.  Assume that   $V\cap \mathcal{H}(r)$ is a finite union of curves and  $\varpi_2(V)$ is dense in $\C^2$. Then for any irreducible curve component $C$ of $V\cap \mathcal{H}(r)$
\[  \gamma (C)\sim 1\]
where the implied constant depends only on $V$.
\end{lemma}

\begin{proof}

By Proposition~\ref{prop:mapsarenice}  the image $\varpi_2(V)$ consists of all of $\C^2(x_2,y_2)$ except perhaps a finite union of curves and points.  The image $ \varpi_2(\mathcal{H}(r))= \mathcal{C}(r)   \subset \C^2(x_2,y_2)$ and since two distinct $\mathcal{C}(r)$ intersect in only a finite set of points, we conclude that $\varpi_2(V \cap \mathcal{H}(r))$ maps onto a dense subset of $\mathcal{C}(r)$ for all but perhaps finitely many $r$.  Therefore, it suffices to show the bound for these $r$ where $\varpi_2(V \cap \mathcal{H}(r))$ is dense in $\mathcal{C}(r)$.

By Proposition~\ref{prop:mapsarenice},  the degree  $d=\deg(\varpi_2\!\! \mid_{V})$ is  finite, and for all but a finite number of points $\{P_m \}$ in the image of $\varpi_2\!\! \mid_{V}$,  the  number of pre-images of any point is bounded by $d$.  For some such $P_m$, the pre-image of $P_m$ is contained in a finite union of curves $\{C_{m,n}\}$ in $V$.   The set $V\cap \mathcal{H}(r)$ is a union of curves and each of these curves intersects a $C_{m,n}$ in   a finite collection of points unless $C_{m,n}$  coincides with a curve component of $V \cap \mathcal{H}(r)$.  
 
 Let $D $ be an irreducible curve component of  $V \cap \mathcal{H}(r)$.
 If $C_{m,n}$ coincides with  $D$ then as there are only finitely many  $C_{m,n}$ the maximal gonality of any such $D$ is bounded by a constant depending only on $V$.  If $C_{m,n}$ does not coincide with $D$ then the degree of $\varpi_2\!\!\mid_{D}$ is bounded by $d$,   which is finite and independent of $r$.      By Lemma~\ref{lemma:abstractgonalitybounds} we conclude that  $\gamma(D) \sim \gamma(\mathcal{C}(r))$. By Lemma~\ref{lemma:keygonality},  $\gamma(\mathcal{C}(r)) = 1$, which concludes the proof.

\end{proof}

\begin{lemma}\label{lemma:key2}Let $V$ be an irreducible surface in $\C^4(x_1,y_1,x_2,y_2)$ with $\dim_{\C}\varpi_i(V)>0$.  Assume that   $V \cap \mathcal{H}(r)$ is a finite union of curves and  $\varpi_2(V)$ is a curve. Then for any irreducible curve component $C$ of $V \cap \mathcal{H}(r)$
 \[ \gamma(C)\sim 1\]
 where the implied constant depends only on $V$.
\end{lemma}

\begin{proof}
The projection $\varpi_2(V \cap \mathcal{H}(r))$ is a finite collection of points unless a curve component of $\mathcal{C}(r)$ is contained in the closure of  $ \varpi_2(V)\subset \C^2(x_2,y_2)$.  This can only occur for finitely many $r$ as $\varpi_2(V)$ is a  curve. For all $r$ in this finite set, $\gamma(V \cap \mathcal{H}(r))$ is  bounded. 

 It now suffices to consider those $r$ such that $\varpi_2(V \cap \mathcal{H}(r))$  is a finite collection of points. 
For these $r$,   any curve  in  $V \cap \mathcal{H}(r)$ must  be of the form $D\times P$ where $D \subset \C^2(x_1,y_1)$ is a curve and $P\subset \C^2(x_2,y_2)$ is a point.
(If $(1,1)$ is in $\varpi_2(V)$ then it is possible that  $P$ are the same for all $r$ since $(1,1)$ is contained in each $D$.)

It suffices to universally bound the gonality of  one such $D_r\times P_r$ as there are only finitely many irreducible components of each $V \cap \mathcal{H}(r).$   Call such a point  $P_r=(a_r,b_r)$.   We use the notation established in \S~\ref{section:preliminaries}.  Let $\{\varphi_i\}$   be a generating set for $\mathfrak{U}_{V}$ where $\varphi_i=\varphi_i(x_1,y_1,x_2,y_2)$.  A generating set for $\mathfrak{U}_{D_r\times P_r}$ is contained in $\{ \varphi_{r,i}^* \}$ where $\varphi_{r,i}^*= \varphi_i(x_1,y_1,a_r,b_r)$.      The degrees of $\varphi_i\!\!\mid_{x_1}$ and $\varphi_i\!\!\mid_{y_1}$ provide an upper bound for the degrees of $\varphi_{r,i}^*$. These upper bounds are independent of $r$.  It follows that  the degree  $d_r$ of $D_r\times P_r$  is bounded by some $d$ which is independent of $r$.

The genus of $D_r\times P_r$ equals the genus of $D_r$ as they are isomorphic. Moreover, the degree of $D_r$ equals the degree of $D_r\times P_r$ as seen by the form of the $\varphi_{r,i}^*$.  Therefore, we identify $D_r\times P_r$ with the plane curve $D_r \subset \C^2(x_1,y_1)$.  By the genus degree formula the genus $g_r$ of $D_r$ satisfies 
\[ g_r \leq \tfrac12 (d_r-1)(d_r-2). \]
(There is equality if the plane curve is non-singular.) The Brill-Noether bound states that for $\gamma_r=\gamma(D_r),$ 
\[ \gamma_r \leq \lfloor \frac{g_r+3}2 \rfloor. \]
Therefore 
\[ \gamma_r \leq \lfloor \frac{g_r+3}2 \rfloor \leq \frac14(d_r-1)(d_r-2)+2.\]
The lemma follows as $d_r\leq d$  and $d$ is independent of $r$.
\end{proof}

\section{Dehn Filling and Character Varieties}

\subsection{} Assume that $M$ is a finite volume hyperbolic 3-manifold
whose boundary consists of exactly $k$ torus cusps. Fix a framing for
each cusp; let $T_i$ be the $i^{th}$ connected component of $\partial
M$ and let $\mu_i$ and $\lambda_i$ be an oriented choice of meridian
and longitude for $T_i$.  For $r_i=p_i/q_i \in \Q \cup \{1/0\}$,
$i=1,\dots, k$, written in lowest terms (with $q$ non-negative), let
$M(r_1,\dots, r_k)$ denote the manifold obtained by performing $r_i$
Dehn filling on the $i^{th}$ cusp of $M$.  We write $r_i=-$ to
indicate that the $i^{th}$ cusp remains complete and is not filled. By
Thurston's Hyperbolic Dehn Surgery theorem, for all but finitely many
values of $r$ in $\Q\cup \{1/0\}$ the manifold obtained by $r$ filling
of one cusp of $M$ is hyperbolic.

We fix a finite  presentation for $\Gamma \cong \pi_1(M)$, with $M_i=[\mu_i]$,  $L_i=[\lambda_i]$ amongst the generators for $i=1,\dots,k$.
If $\Gamma$ has presentation 
\[ \Gamma = \langle \gamma_1,\dots, \gamma_n : \varrho_1, \dots \varrho_l \rangle\]
then by van Kampen's theorem, a presentation for \[ \Gamma(r_1,\dots,r_k) \cong  \pi_1(M(r_1,\dots, r_k))\] is given by 
\[ \Gamma(r_1, \dots, r_k) = \langle \gamma_1, \dots, \gamma_n : \varrho_1, \dots \varrho_l, w_1,\dots, w_k \rangle\]
where $w_i=M_i^{p}L_i^q$ unless $r=p/q=-$ in which case it is empty.

We state the next lemma for convenience.  The proof is clear.

\begin{lemma}\label{lemma:homologyfillingmap} Let $M$ be a finite volume two-cusped hyperbolic 3-manifold.  Then 
\[ |H^1(M(-,r); \Z/2\Z)| \leq |H^1(M;\Z/2\Z)| \] for all $r\in \Q \cup \{1/0\}$.
\end{lemma}


\subsection{} 

We will now define various varieties associated to the fundamental
group of a 3-manifold $M$ (see 
\cite{MR2104008}.)  For a finitely generated group $\Gamma$, and a homomorphism $\rho:\Gamma  \rightarrow
\SL_2(\C),$
the
set \[ \text{Hom}(\Gamma, \SL_2(\C)) =\{ \rho:\Gamma \rightarrow
\SL_2(\C) \}\] naturally carries the structure of an affine algebraic
set defined over $\C$.  This is called the {\em representation
  variety}, $R(\Gamma)$.  Given an element $\gamma \in \Gamma$ we
define the trace function $I_{\gamma}:R(\Gamma) \rightarrow \C$ as
$I_{\gamma}(\rho) = \tr(\rho(\gamma))$.  Similarly, given a
representation $\rho \in R(\Gamma)$ we define the character of $\rho$
to be the function $\chi_{\rho}:\Gamma \rightarrow \C$ given by
$\chi_{\rho}(\gamma) = \tr(\rho(\gamma))$.  The {\em
  $\SL_2(\C)$-character variety} of $\Gamma$ is the set of all
characters, \[X(\Gamma) = \{\chi_{\rho}: \rho\in R(\Gamma)\}.\] Such a
set is called a {\em natural model} for the character variety, as
opposed to a set isomorphic to one obtained in this manner. The functions 
$t_{\gamma}:R(\Gamma) \rightarrow X(\Gamma)$ defined by
$t_{\gamma}(\rho)=\chi_{\rho}(\gamma)$ are regular maps.  The
coordinate ring of $X(\Gamma)$ is $T\otimes \C$ where $T$ is the
subring of all regular functions from $R(\Gamma) $ to $\C$ generated
by 1 and the $t_{\gamma}$ functions.

If $\Gamma$ is isomorphic to $\Gamma'$, then the sets $R(\Gamma)$ and $R(\Gamma')$ are isomorphic, as are $X(\Gamma)$ and $X(\Gamma')$.  We write $R(M)$ or $X(M)$ to denote a specific $R(\pi_1(M))$ or $X(\pi_1(M))$, respectively,  which is well-defined up to isomorphism.   The inclusion homomorphism $i: \pi_1(\partial M) \rightarrow \pi_1(M)$ induces a map $i^*:X(M)\rightarrow X(\pi_1(\partial M))$ by $\chi_{\rho} \mapsto \chi_{\rho \mid_{\pi_1(\partial M)}}.$  We write $A(M)$ to denote the closure of the image of this map.

The representation variety $R(M)$ naturally sits in $\C^{4n}(g_{11},g_{12},g_{13},g_{14}, \dots, g_{n4})$ where $g_{i1}, g_{i2}$, $g_{i3}$ and $g_{i4}$ are the four entries of the $2\times 2$ matrix $\rho(\gamma_i)$.   Fix a generating set for $\pi_1(\partial M)$, $\{ M_1, L_1, \dots, M_k,L_k\}$ where each pair $\{M_i,L_i\}$ generates $\pi_1(T_i)$, the fundamental group a connected component of $\partial M$.   Let $m_i^{\pm 1}$ be the eigenvalues of $\rho(M_i)$ and $l_i^{\pm 1}$ be the eigenvalues of $\rho(L_i)$.

Define the following functions on $\C(g_{11}, \dots, g_{n4}, m_1^{\pm 1}, l_1^{\pm 1}, \dots, m_k^{\pm 1} ,l_k^{\pm 1} )$
\[ I_{M_i} - (m_i+m_i^{-1}), \quad I_{L_i}-(l_i+l_i^{-1}), \quad
I_{M_iL_i} -(m_il_i+m_i^{-1}l_i^{-1}).\] The extended representation
variety, $R_E(M)$ is the algebraic set in $\C^{4n} \times(\C^*)^{2k}$
cut out by $\mathfrak{U}_{R(M)}$ and the equations above for $i=1,
\dots, k$. The extended character variety $X_E(M)$ is defined
similarly.  The {\em eigenvalue variety}, $E(M)$, is the closure of
the image of $R_E(M)$ under the projection to the coordinates $m_1,
l_1, \dots, m_k, l_k$.  It naturally sits in $(\C^*)^{2k}(m_1^{\pm 1},
l_1^{\pm1}, \dots, m_k^{\pm 1}, l_k^{\pm 1})$.   We have the following commuting
diagram where the maps $p_1$, $p_2$, and $p_3$ are all finite to one of
the same degree.

\hspace{3cm}
\xymatrix{ R_E(M) \ar[r]^{t_E} \ar[d]_{p_1} & X_E(M) \ar[r]^{i_E^*} \ar[d]_{p_2} & E(M) \ar[d]_{p_3} \\
R(M) \ar[r]_{t} & X(M) \ar[r]_{i^*} & A(M)  }

An irreducible component of $R(M)$ containing a discrete and faithful representation is denoted $R_0(M)$, and the image in $X(M)$ is denoted $X_0(M)$ and called a {\em canonical component}. For a given $M$ there can be more than one canonical component, as there are multiple lifts of a discrete faithful representation to $\SL_2(\C)$.  We call an irreducible component of $A(M)$ a canonical component and write $A_0(M)$ if it is the image of a canonical component in $X(M)$. This definition also extends to the extended varieties.

We can define the $\PSL_2(\C)$ representation and character varieties in a similar way. (See \cite{MR1670053}.)  We denote the $\PSL_2(\C)$ character variety  of $M$ by $Y(M)$ and a canonical component by $Y_0(M)$.  We can form the restriction variety $B(M)$ analogous to $A(M)$ as above (and also form $B_0(M)$  when appropriate).  
The group $\PSL_2(\C)$ is isomorphic to $\rm{Isom}^+(\H^3)$, the group of orientation preserving isometries of hyperbolic half space.  
Since the index $[\rm{Isom}(\H^3):\rm{Isom}^+(\H^3)]=2$, for $M$ a finite volume hyperbolic
3-manifold with at least one cusp there are two $\PSL_2(\C)$ characters
corresponding to discrete faithful representations of $\pi_1(M)$. These correspond to different orientations of $M$.  The authors know of no examples where these two characters lie on different irreducible components of $Y(M)$.
 The coordinate ring of the quotient is $T_e\otimes \C$ where $T_e$ is the subring of $T$ consisting of all elements invariant under the action of $\mu_2$.   
Culler  \cite{MR825087}  (Corollary 2.3) showed that a faithful representation of any discrete torsion-free subgroup to $\PSL_2(\C)$ lifts to $\SL_2(\C)$, and moreover (Theorem 4.1) that a representation $\rho:\Gamma \rightarrow \PSL_2(\C)$ lifts to a representation to $\SL_2(\C)$ if and only if every representation in the path component containing $\rho$ lifts.

  By the proof of Thurston's Dehn Surgery Theorem
(\cite{Thurston} Theorem 5.8.2), any neighborhood of the character of
a discrete faithful representation in $Y(M)$ contains all but finitely
many Dehn filling characters.  Since a discrete faithful character is always a simple point of
$Y(M)$ \cite{MR1396960} it follows $Y_0(M(-,r))$ is contained in some $ Y_0(M)$ for all but
finitely many $r$.  Therefore, for all but finitely many $r$,
$X_0(M(-,r))$ is contained in some canonical component of $X(M)$. (Due
to lifting restrictions it is not necessarily true that all but
finitely many $X_0(M(-,r))$ are contained in a given $X_0(M)$.)

In the light of Lemma~\ref{lemma:abstractgonalitybounds} we now
establish bounds on the degrees of maps between the varieties of
interest.

\begin{lemma}\label{lemma:degrees} Let $M$ be a finite volume hyperbolic 3-manifold with exactly two cusps. Assume that $M(-,r)$ is hyperbolic. 
\begin{enumerate}

\item\label{lemma:degrees2} $Y_0(M(-,r) \rightarrow B_0(M(-,r))$ has degree 1.  
\item\label{lemma:degrees3} $X_0(M(-,r)) \rightarrow A_0(M(-,r))$ has degree at most   $\tfrac12|H^1(M;\Z/2\Z)|$.
\item\label{lemma:degrees4} $E_0(M(-,r)) \rightarrow A_0(M(-,r))$ has degree  4. 
\item\label{lemma:degrees5} $X_0(M(-,r)) \rightarrow Y_0(M(-,r))$ has degree at most $|H^1(M;\Z/2\Z) |$.   
\end{enumerate}
These mappings are all dense. 
\end{lemma}

\begin{proof}

  Part (\ref{lemma:degrees2}) follows directly from Dunfield
  \cite{MR1695208} (Theorem 3.1), who showed that the mapping
  $i^*:Y_0(M(-,r)) \rightarrow B_0(M(-,r))$ is a birational map onto
  its image.  Dunfield \cite{MR1695208} (Corollary 3.2), also proved
  that the degree of the map $i^*:X_0(M(-,r)) \rightarrow A_0(M(-,r))$
  is at most $\tfrac12 |H^1(M(-,r);\Z/2\Z)|$.  By
  Lemma~\ref{lemma:homologyfillingmap} this is bounded by
  $\tfrac12 |H^1(M;\Z/2\Z) |$.  This establishes part (\ref{lemma:degrees3}).
  The mapping from $E_0(M(-,r))$ to $X_0(M(-,r))$ is given by the
  action of inverting both entries of the pair $(m_1,l_1)$ to
  $(m_1^{-1}, l_1^{-1})$.  This has degree at most four, giving part
  (\ref{lemma:degrees4}).  
  Part (\ref{lemma:degrees5}) follows
  from the observation that the number of lifts of a representation
  $\rho:\Gamma \rightarrow \PSL_2(\C)$ to $\SL_2(\C)$ is
  $|H^1(\Gamma;\Z/2\Z)|$.  Therefore the degree of the map is at most
  $ |H^1(M(-,r); \Z/2\Z)|$ which is bounded above by $|H^1(M;\Z/2\Z)|$
  by Lemma~\ref{lemma:homologyfillingmap}. 
  since the maps commute.

The mappings must all have dense image in the target space 
since all maps considered are maps between curves and the image must have dimension one as the degrees of the maps are bounded. 
\end{proof}

Given Lemma \ref{lemma:degrees}, and Definition~\ref{definition:sim} we record the
following for convenience.

\begin{prop}\label{prop:relatedvarieties}
Let $M$ be a finite volume hyperbolic 3-manifold with exactly two cusps.  Assume that $M(-,r)$ is hyperbolic. Define the set 
\[ \mathfrak{S}= \{ A_0(M(-,r)),   B_0(M(-,r)), E_0(M(-,r)), X_0(M(-,r)), Y_0(M(-,r))  \}.\]
For any $V, W \in \mathfrak{S}$, $\gamma(V) \sim \gamma(W)$.
\end{prop}

  \begin{proof}
  By Lemma~\ref{lemma:abstractgonalitybounds} it suffices to show that for all but finitely many $r$ that  there is a  dominant map from $V$ to $W$ or from $W$ to $V$ whose degree is independent of $r$.   The lemma now follows from Lemma~\ref{lemma:degrees}.
\end{proof}

\section{Isolation of cusps}\label{section:Isolationofcusps}

Let $M$ be a finite volume hyperbolic 3-manifold with $k$ cusps.   Geometric isolation was studied extensively in   \cite{MR1202390}.  

\begin{definition}

Cusps $1,\dots, l$ are {\em geometrically isolated} from cusps $l+1,\dots k$ if any deformation induced by Dehn filling cusps $l+1, \dots k$ while keeping cusps $1,\dots l$ complete does not change the Euclidean structures at cusps $1,\dots, l$.

Cusps $1,\dots, l$ are {\em first order isolated} from cusps $l+1,\dots, k$ if the map from the space of deformations induced by Dehn filling cusps $l+1,\dots, k$ while keeping cusps $1,\dots, l$ complete to the space of Euclidean structures at cusps $1,\dots, l$ has zero derivative at the point corresponding to the complete structure on $M$.

Cusps $1,\dots, l$ are {\em strongly geometrically isolated} from cusps $l+1, \dots, k$ if performing integral Dehn filling of cusps $1,\dots, l$ and replacing the cusps by goedesics $\gamma_1,\dots, \gamma_l$ and then deforming cusps $l+1,\dots, k$ does not change the geometry of $\gamma_1,\dots, \gamma_l$. 
\end{definition}

Strong geometric isolation implies geometric isolation which implies first order isolation. Strong geometric isolation and first order geometric isolation are symmetric conditions, but geometric isolation is not necessarily symmetric.

We now restrict to our case of interest, namely the case
when $M$ has exactly two cusps, $T_1$ and $T_2$.  The projection $\varpi_i(A_0(M))$ is either a curve or is dense in $\C^2$ (see Lemma~\ref{lemma:containment}
below).
This projection is a curve exactly when the projection $\varpi_i(B_0(M))$ is a curve, and therefore this dimension is well-defined even if there are multiple $A_0(M)$ components.

We will adopt the notation from  \cite{MR1202390}. Define the set \[ S= \{ (u_1,\tau_1(u_1,u_2),u_2, \tau_2(u_1,u_2))\}\subset \C^4(u_1,w_1,u_2,w_2)\] where  $u_i $ is twice the logarithm of the eigenvalue of the holonomy of the meridian of $T_i$ and $v_i$ is twice the logarithm of the  eigenvalue of the holonomy of the longitude of $T_i$,   and $\tau_i=u_i/v_i$.

\begin{lemma}\label{lemma:geoiso} The first cusp of $M$  is geometrically isolated from the second cusp of $M$ if and only if  $\varpi_1(A_0(M))$ is a curve.
\end{lemma}

\begin{proof}

The $u_i$ coordinates correspond to the meridional $m_i$ coordinates. The $\tau_i$ coordinates are of the form $u_i/v_i$ and therefore (as long as $u_iv_i \neq 0$) are in bijective correspondence with the $v_i$ coordinates.  These, in turn, give the longitudinal coordinates, $l_i$.  Therefore, the dimension of $A_0(M)$ is the same as the dimension of $S$, and the dimension of $\varpi_i(A_0(M))$ equals the dimension of $\varpi_i(S)$. 

The projection $\varpi_1(A_0(M))$  is $\{(u_1,\tau_1(u_1,u_2))\}$.  This is a curve if and only if $\tau_1$ is a function of $u_1$ alone, and not a function of $u_2$.  In this case $\tau_1(0,u_2)$ is constant.  This is the condition for $T_1$ to be geometrically isolated from $T_2$ (see  \cite{MR1202390}, proof of Theorem 4.2). 

Assume that $T_1$ is geometrically isolated from $T_2$.    The projection $\varpi_1(A_0(M))$ is a curve exactly when $\{(u_1,\tau_1(u_1,u_2))\}$ is a curve.  In this case,  by \cite{MR1202390} $\tau_1(0,u_2)$ is a constant, $c$.  Therefore the line $\{(0,w_1)\} \subset \C^2(u_1,w_1)$ 
 intersects $\{(u_1,\tau_1(u_1,u_2))\}$ in a single point, $(0,c)$.   It follows that $\{(u_1,\tau_1(u_1,u_2))\}$ cannot be dense in $\C^2$.  (A Riemannian neighborhood of any point in $\{(u_1,\tau(u_1,u_2))\}$ is homeomorphic to $\C^2$ if the image is dense.  However, by the above, a neighborhood of $(0,c)$ cannot be homeomorphic to $\C^2$.) Therefore, the image has dimension 1. As $A_0(M)$ is irreducible by construction, by Lemma~\ref{lemma:irreducible} we conclude that $\varpi_1(A_0(M))$ is a curve as well.

\end{proof}

\begin{lemma}\label{lemma:stronggeoiso}
The cusps of $M$ are strongly geometrically isolated from one another if and only if $A_0(M) \cong C_1 \times C_2$ where $C_i$ is a curve in $\C^2(m_i,l_i)$.  
\end{lemma}

\begin{proof}
  By \cite{MR1202390} Theorem 4.3, $T_1$ is strongly geometrically
  isolated from $T_2$ if and only if $v_1$ is dependent on $u_1$ and
  not $u_2$.  Strong geometric isolation is a symmetric condition (see
  \cite{MR1202390}). Therefore, $T_1$ is strongly geometrically
  isolated from $T_2$ if and only if $m_1$ depends only on $l_1$ and
  $m_2$ depends only on $l_2$. Therefore, for any irreducible
  polynomial $f$ in the vanishing ideal of $A_0(M)$, $f$ is a
  polynomial of $m_1$ and $l_1$ or of $m_2$ and $l_2$.  It follows
  that $A_0(M)$ is the product of curves if and only if $T_1$ is
  strongly geometrically isolated from $T_2$.

\end{proof}

\begin{cor}
The cusps of $M$ are strongly geometrically isolated from one another if and only if both cusps are geometrically isolated from one another.
\end{cor}

\begin{proof}
If the cusps of $M$ are strongly geometrically isolated, then by Lemma~\ref{lemma:stronggeoiso} $A_0(M)\cong C_1\times C_2$ is a product of curves, so that $\varpi_i(A_0(M))=C_i$ is a curve for $i=1,2$ and by Lemma~\ref{lemma:geoiso}  both cusps are geometrically isolated from each other.  If both cusps are geometrically isolated from each other, then $\varpi_i(A_0(M))$ is a curve, say $C_i$ for $i=1,2$.  The surface $A_0(M)$  is contained in both $\varpi_1(C_1)^{-1}= C_1 \times \C$ and $\varpi_2(C_2)^{-1}= \C\times C_2$. The intersection of these sets is $C_1\times C_2$, so $A_0(M)$ is a product of curves. By Lemma~\ref{lemma:stronggeoiso} we conclude that the cusps of $M$ are strongly geometrically isolated.

\end{proof}

\section{Gonality in Dehn surgery space}\label{section:gonalityindehnsurgeryspace}

In this section we restrict to the case where $M$ is a finite volume hyperbolc 3-manifold with exactly two cusps. 
Before we proceed, we define some Chebyshev polynomials which will appear as traces.
\begin{definition}\label{definition:fib}
For any integer $k$ let $f_k(x)$ be the $k^{th}$ Fibonacci
  polynomial, that is defined $f_0(x)=0$, $f_1(x)=1$ and for all other
  $k$, define $f_k(x)$ recursively by the
  relation \[f_{k+1}(x)+f_{k-1}(x)= xf_k(x).\] Define
  $g_k=f_k-f_{k-1}$, and $h_k(x)=f_{k+1}-f_{k-1}$.
\end{definition}

\begin{remark}
With $x=s+s^{-1}$ then 
 \[ f_k(x) = \frac{s^k-s^{-k}}{s-s^{-1}}, \quad 
 g_k(x) = \frac{s^k+s^{1-k}}{s+1}, \quad \text{and} \quad 
 h_k(x) = s^k+s^{-k}.
 \] 

\end{remark}

The effect of  $r=p/q$ Dehn filling of the second cusp of $M$ on the fundamental group is that it introduces the group relation  $M_2^pL_2^q=1$.  Therefore, for a representation $\rho$ the relation introduces an additional  matrix relation  $\rho(M_2)^p\rho(L_2)^q=I_2$ where $I_2$ is the $2\times 2$ identity matrix.  Up to conjugation, $\rho(M_2)$ can be taken to be upper triangular, and since $M_2$ and $L_2$ commute, $\rho(L_2)$  is upper triangular as well. (If $\rho(M_2)=I_2$ we can conjugate so that $\rho(L_2)$ is upper triangular.)  The character and eigenvalue varieties  are not changed by conjugation.  Therefore, we may take  
\[ \rho(M_2) = \mat{m}{s}{0}{m^{-1}} \quad \text{and} \quad \rho(L_2)=\mat{l}{t}{0}{l^{-1}}\]
where $m,l\in \C^*$ and $s,t\in \C$.  
The commuting condition ensures that
\begin{equation}\label{eqn:1} t(m-m^{-1})=s(l-l^{-1}). \end{equation}
By the Cayley-Hamilton theorem, 
\[ \rho(M_2)^p = \mat{m^p}{sf_p(m+m^{-1})}{0}{m^{-p}} \quad \text{and} \quad \rho(L_2)^{-q}=\mat{l^{-q}}{-t f_q(l+l^{-1})}{0}{l^{q}}.\]
As $\rho(M_2)^p=\rho(M_2)^{-q}$,  we conclude that 
\[ m^pl^q=1 \ \text{ and } \  sf_p(m+m^{-1})=-t f_q(l+l^{-1}). \]
If neither $m$ nor $l$ is $\pm 1$  then \eqref{eqn:1} implies that 
\[ 
s \frac{m^p-m^{-p}}{m-m^{-1}} = t \frac{m^p-m^{-p}}{l-l^{-1}}.
\]
With $m^pl^q=1$ the condition  $sf_p(m+m^{-1})=-t f_q(l+l^{-1})$ follows.  Therefore, the relation $sf_p(m+m^{-1})=-t f_q(l+l^{-1})$ follows from \eqref{eqn:1}  and $m^pl^q=1$ unless $m$ or $l$ is $\pm 1$.  If $m=\pm 1$ or $l=\pm 1$, an additional identity with $s$ and $t$ is required to deduce $sf_p(m+m^{-1})=-t f_q(l+l^{-1})$ from the commuting condition and $m^pl^q=1$. 
(The commuting condition \eqref{eqn:1} is part of the defining equations for $X(M)$ as $M_i$ and $L_i$ commute in the group.)
In the next section, we will see how this additional relation manifests as an intersection in the character varieties.  First, we will  discuss the containment of these sets.

The set $X(\partial M)$ is contained in $\C^4(m_1,l_1,m_2,l_2)$ where the coordinates are distinguished eigenvalues of the meridians and longitudes of the cusps. 
Any irreducible component has
dimension at most two, and the image of a canonical component
has dimension two (See \cite{MR1886685}  and  \cite{MR2104008} Proposition 12). The set $A( M(-,r))$ is contained in $\C^2(m,l)$
where $m$ and $l$ are identified with $m_1$ and $l_1$, and any
canonical component has dimension one.

For all but finitely many $r,$ $Y_0(M(-,r)) $ is contained in some $ Y_0(M)$.  Therefore for all but finitely many $r$,
$X_0(M(-,r))$ is contained in some $X_0(M)$.  There may be numerous
lifts of $Y_0(M)$, each containing infinitely many $X_0(M(-,r))$. 
 As $A_0(M)\subset \C^4$ and $A_0(M(-,r)) \subset \C^2$ it is not the case that $A_0(M(-,r))\subset A_0(M)$.   In Lemma~\ref{lemma:containment} we will show the relationship between these two sets.

\subsection{Filling Relations}\label{section:FillingRelations}

 Recall the notation established in \S\ref{section:DefinitionsandAlgebraicGeometry}; if $f$ is a polynomial, $V(f)$ indicates the vanishing set of (the ideal generated by) $f$ in a specified $\C^n$.  We will make extensive use of the sets and polynomials from  Definition~\ref{definition:psi}.  We will write $V(m_2^{\pm p}l_2^q-1)$ for $ V(m_2^{|p|}l_2^q-1) \cup V(m_2^{|p|}-l_2^q)$.

\begin{lemma}\label{lemma:Xcharvarinclusion}  Assume that  $M(-,r)$ is hyperbolic.  For all but finitely many $r$ there is some $X_0(M)$ such that 
\[ X_0(M(-,r)) \subset   X_0(M) \cap \Big( V(m_2^{\pm p}l_2^q-1) \Big).\]
\end{lemma}

\begin{proof} 
For all but perhaps finitely many $r$ such that $M(-,r)$ is hyperbolic there is some $X_0(M)$ containing a given $X_0(M(-,r))$. Consider these $r=p/q$.
The fundamental group $\pi_1(M(-,r))$ is the quotient of $\pi_1(M)$ corresponding to the addition of a single relation, $M_2^pL_2^q=1$.  Therefore, the trace of $\rho(M_2)^p$ equals the trace of  $\rho(L_2)^{-q}$.  By construction, these traces are $h_p(\chi_\rho(M_2))$ and $h_q(\chi_\rho(L_2))$ with $\chi_\rho(M_2)=m_2+m_2^{-1}$ and $\chi_\rho(L_2)=l_2+l_2^{-1}$. This shows that 
\[
 X_0(M(-,r)) \subset    X_0(M) \cap V\big( h_p(m_2+m_2^{-1})-h_q(l_2+l_2^{-1})\big).
\]
It suffices to show that 
\[ 
V\big( h_p(m_2+m_2^{-1})-h_q(l_2+l_2^{-1})\big) \subset V(m_2^{\pm p}l_2^q-1).
\]
This is equivalent to showing that the solution set of $m_2^p+m_2^{-p}= l_2^q+l_2^{-q}$ is contained in the solution set of $m_2^{\pm p}l_2^q=1$ in $\C^4(m_2,l_2)$.  This follows directly from the observation that the  solutions to an equation of the form $x + x^{-1} =  y + y^{-1}$ are $ x = y^{\pm 1}$.    

 \end{proof}

\begin{remark}
It is possible that other trace relations will specify the sign in the relation $m_2^{\pm p}l_2^q -1$ but this is not necessarily the case.  
\end{remark}

We will assume that $M$ is hyperbolic and $X_0(M(-,r)) \subset
X_0(M)$.  The sets $A_0(M(-,r))$ and $A_0(M)$ lie in different ambient
spaces, but we have the following lemma which allows us to concretely
relate $A_0(M(-,r))$ with $A(M)$.  Since $M$ is hyperbolic $A_0(M)$ is
a surface, and for all but finitely many $r$, $A_0(M(-,r))$ is a
curve. As a result, the projection map $\varpi_i(A_0(M))$ has
dimension at most two.  The following lemma will allow us to deduce that
$\varpi_i(A_0(M))$ cannot be a point, and therefore is either a curve
or is dense in $\C^2$. (The image is irreducible by
Lemma~\ref{lemma:irreducible}.)  Our analysis will depend upon the
dimension of these images.

\begin{lemma}\label{lemma:containment}
Assume that  $M(-,r)$ is hyperbolic.  For all but finitely many $r$,  $A_0(M) \cap \mathcal{H}(r)$ is a finite union of curves and
the variety $A_0(M(-,r))$ is   a curve component  of the closure of $\varpi_1( A_0(M) \cap \mathcal{H}(r))$. \end{lemma}

\begin{proof} 

The group $\pi_1(M(-,r))$ is a quotient of $\pi_1(M)$; the group presentation only has the additional relation that $M_2^pL_2^q=1$, where $r=p/q$.  Using the notation established earlier in this section, in the construction of $X(\partial M)$ and $A( M(-,r))$ the $(1,1)$ entry of $\rho(M_i)$  is identified with $m_i$, and the $(1,1)$ entry for $\rho(L_i)$ is identified with $l_i$.  The manifestation of this group quotient on the character varieties is that   $X_0(M(-,r))$ is  contained in the intersection of $X_0(M)$ and  $V(m_2^{\pm p}l_2^q-1)$, as shown in Lemma~\ref{lemma:Xcharvarinclusion}. 

The epimorphism $\pi_1(M)\rightarrow \pi_1(M(-,r))$ surjects $\pi_1(\partial M)$
onto $\pi_1(\partial M(-,r))$ and this induces maps 

\[ X(M) \overset{i^*}{ \rightarrow} X(\partial M) \overset{\varpi_1}{\rightarrow} X(\partial M(-,r)).\]

The set $X(\partial M)$ is dense in  $\C^4(m_1,l_1,m_2,l_2)$ and the set $X(\partial M(-,r))$ is contained in $\C^2(m_1,l_1)$.

Since $X_0(M(-,r)) \subset X_0(M) \cap V(m_2^{\pm p}l_2^q-1)$, it follows that  
\[ 
i_*(X_0(M(-,r))) \subset i_*(X_0(M)) \cap i_*(V(m_2^{\pm p}l_2^q-1))).
\]
Therefore, $i_*(X_0(M(-,r)))  \subset A_0(M) \cap V(m_2^{\pm p}l_2^q-1)$ in $\C^4$.  Finally,  since  $A_0(M(-,r))$ is the closure of $\varpi_1(i_*(X_0(M(-,r))))$, by projecting we conclude that 
$ A_0(M(-,r)) $ is contained in (the closure of) $\varpi_1(A_0(M) \cap V(m_2^{\pm p}l_2^q-1))$, which is  $\varpi_1(A_0(M) \cap \mathcal{H}(r)).$

To show that $A_0(M) \cap \mathcal{H}(r)$ is a finite union of curves, 
by Lemma~\ref{lemma:V(r)curves} it suffices to show  $\dim_{\C}(\varpi_i(A_0(M)))>0$ for $i=1,2$.     Since $M(-,r)$ is hyperbolic and has exactly one cusp $A_0(M(-,r))$ is a curve.   From the above,    \[ A_0(M(-,r)) \subset \varpi_1(A_0(M) \cap \mathcal{H}(r)) \subset \varpi_1(A_0(M))\]   which  implies that 
$\dim_{\C} \varpi_1(A_0(M))>0$.   By symmetry (filling the other cusp), $\dim_{\C} \varpi_2(A_0(M))>0$ as well.

\end{proof}

\begin{figure}[h!]
\vspace{-0.5cm}
\begin{center}
\includegraphics[scale=.4]{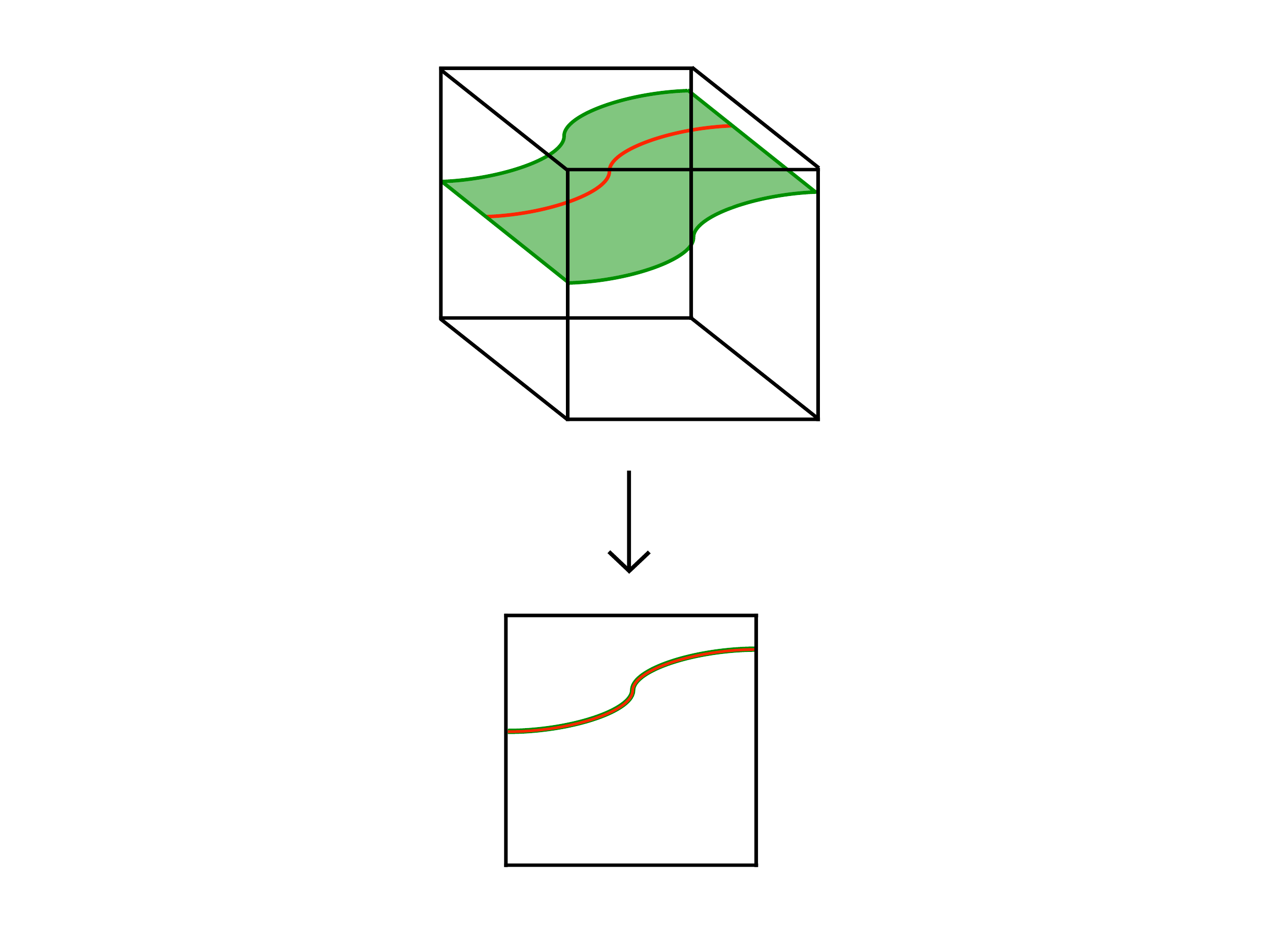}
\end{center}
\vspace{-0.8cm}
\caption{The set $A_0(M)\subset \C^4(m_1,l_1,m_2,l_2)$ with the intersection $A_0(M)\cap  \mathcal{H}(r)$ and the  projection $\varpi_1(A_0(M)) \subset \C^2(m_1,l_1)$ when the first cusp is geometrically isolated from the second. The set $A_0(M(r))\subset \varpi_1(A_0(M))$.}
\label{fig:ProjectionGeoIso}
\end{figure}

The following lemma shows that if the first cusp of $M$ is geometrically isolated from the second cusp of $M$ then there are only finitely many $A_0(M(-,r))$.   We conclude that in this situation,  the gonality of any $A_0(M(-,r))$ is bounded independently of $r$.

\begin{lemma}\label{lemma:geoisosamecurve}
Assume that  the first cusp of $M$ is geometrically isolated from second cusp of $M$.  For all but finitely many $r$ such that $A_0(M(-,r))\subset \varpi_1(A_0(M))$,  all of the  curves $A_0(M(-,r))$ are identical.
\end{lemma}

\begin{proof}
 Lemma~\ref{lemma:containment} implies that for all but finitely many $r$, $A_0(M(-,r))$ is contained in $\varpi_1(A_0(M) \cap \mathcal{H}(r))$.  For all such $r$, $A_0(M(-,r))$ is a curve and so is $\varpi_1(A_0(M))$ by Lemma~\ref{lemma:geoiso}.  It follows that they must be the same curve.
\end{proof}

\begin{remark} A similar statement applies to $B_0(M)$.
If there is a single canonical component $Y_0(M)$ then Lemma~\ref{lemma:geoisosamecurve} shows that for all but finitely many $r$ the curves $B_0(M(-,r))$ are the same curve. 
\end{remark}

\subsection{Gonality Relations}\label{section:GonalityRelations}

Now we show that if the first cusp of $M$ is not geometrically isolated from the second cusp of $M$, in order to bound the gonality of $A_0(M(-,r))$ it suffices to bound the gonality of $A_0(M) \cap \mathcal {H}(r)$.

\begin{figure}[h!]
\begin{center}
\includegraphics[scale=.4]{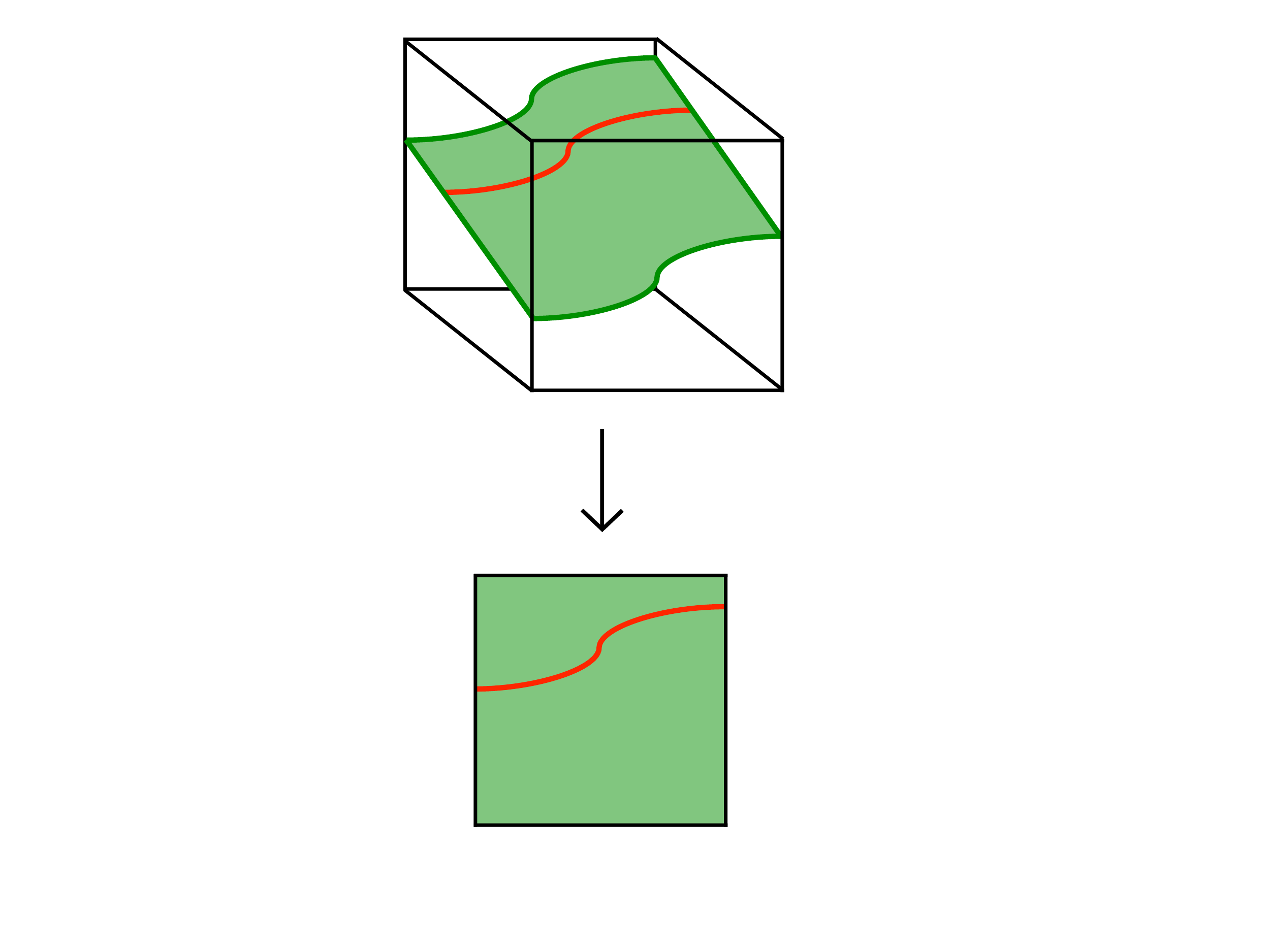}
\end{center}
\vspace{-1cm}
\caption{The set $A_0(M)\subset \C^4(m_1,l_1,m_2,l_2)$ with the intersection $A_0(M)\cap  \mathcal{H}(r)$ and the  projection $\varpi_1(A_0(M)) \subset \C^2(m_1,l_1)$ when the first cusp {\em is not} geometrically isolated from the second. The set $A_0(M(r))\subset \varpi_1(A_0(M))$.}
\label{fig:ProjectionNonGeoIso}
\end{figure}

\begin{lemma}\label{lemma:firstprojection}
Assume that  the first cusp of $M$ is not geometrically isolated from the second cusp of $M$.  Then for 
all but finitely many $r$ such that $A_0(M(-,r))  \subset \varpi_1(A_0(M))$ the set $A_0(M) \cap \mathcal {H}(r)$ is a finite union of curves.  Additionally,  for
any irreducible curve component $C$ of $A_0(M) \cap \mathcal {H}(r)$ with 
 $A_0(M(-,r))$ contained in the closure of $\varpi_1(C)$ 
\[ \gamma(C) \sim \gamma(A_0(M(-,r))).\]
\end{lemma}

\begin{proof}  It suffices to prove the lemma for a single irreducible component  $A_0(M)$ whose projection contains infinitely many $A_0(M(-,r))$. 
By Lemma~\ref{lemma:abstractgonalitybounds} it suffices to show that with the exception of finitely many $r$, the degree of $\varpi_1$ restricted to 
$A_0(M)\cap \mathcal{H}(r)$ is independent of $r$.

By  Lemma~\ref{lemma:geoiso}   since  the first cusp of $M$ is not geometrically isolated from the second cusp of $M$, the map $\varpi_1:A_0(M)\rightarrow \C^2$ is dense.   By Proposition~\ref{prop:mapsarenice} this 
is a dense map between surfaces of degree $d$.  Moreover, for all but finitely many points in the image of $\varpi_1(A_0(M))$ the pre image of a point consists of at most $d$ points in $A_0(M)$.  The image $\varpi_1(A_0(M))$ is $\C^2$ with the possible exception of finitely many  curves and points.  By  Lemma~\ref{lemma:containment} except for  finitely many points, $A_0(M(-,r))$ is contained in the image of $\varpi_1(A_0(M))$.   

Therefore, for all but finitely many points $P$ on any curve $A_0(M(-,r))$, the number of pre-images  of $P$ under $\varpi_1$ is at most $d$. 
 As a result, if $C\subset \big( A_0(M) \cap \mathcal {H}(r) \big)$ is an irreducible component such that  $A_0(M(-,r))$ is contained in the closure of $\varpi_1(C)$, the map $\varpi_1 \mid_C$ is degree at most $d$.  Since $d$ independent of $r$, we have $\gamma(C) \sim \gamma(A_0(M(-,r)))$.

\end{proof}

\subsection{Proof of Theorem~\ref{theorem:1} }

\begin{proof}[Proof of Theorem~\ref{theorem:1}] To prove Theorem
  ~\ref{theorem:1} it follows from
  Proposition~\ref{prop:relatedvarieties} that it suffices to bound
  the gonality of $A_0(M(-,r))$.  For all but finitely many $r,$
  $X_0(M(-,r))$ is contained in some $X_0(M)$.  Therefore, for all but
  finitely many $r,$ $A_0(M(-,r))$ is contained in some
  $\varpi_1(A_0(M))$.  As there are only finitely many $A_0(M)$, it
  suffices to consider all $r$ such that $A_0(M(-,r))$ is contained in
  a fixed $\varpi_1(A_0(M))$.

Assume that neither cusp of $M$ is geometrically isolated from the other. By Lemma~\ref{lemma:geoiso}, $A_0(M)$ is a surface in $\C^4$ whose image under both $\varpi_1$ and $\varpi_2$ is dense in $\C^2$.  By Lemma~\ref{lemma:firstprojection} $A_0(M(-,r))$ is contained in the closure of  $\varpi_1(C)$ for some irreducible curve component $C$ of $A_0(M) \cap \mathcal {H}(r)$, and $\gamma(C) \sim \gamma(A_0(M(-,r))$.  By Lemma~\ref{lemma:key2} since $C \subset A_0(M) \cap \mathcal{H}(r)$ and $\varpi_2(A_0(M))$ is dense in $\C^2$ it follows that  $\gamma(C) \sim 1$.  By the transitivity of $\sim$ we conclude that $\gamma(A_0(M(-,r)))\sim 1$.

Now we consider the case where there is one cusp geometrically isolated from another; one of the projections $\varpi_i(A_0(M))$ is a curve by Lemma~\ref{lemma:geoiso}.

Assume that $\varpi_1(A_0(M))$ is a curve, $D$, and $A_0(M(-,r))\subset \varpi_1(A_0(M))$. That is assume the first cusp is geometrically isolated from the second cusp.  By Lemma~\ref{lemma:containment}, for all but finitely many $r$,  $A_0(M(-,r))$ is  a curve component of the closure of $\varpi_1(A_0(M) \cap \mathcal{H}(r))$.   We conclude that for  all but finitely many of the $r$ above, $\varpi_1(A_0(M) \cap \mathcal{H}(r))=D$ as well, since  this projection is a union of curves contained in $D$.   (In this case the map $\varpi_1\mid_{A_0(M)}$ does not have finite degree  and therefore we cannot conclude that $\gamma(A_0(M(-,r))) \sim \gamma(A_0(M) \cap \mathcal{H}(r))$.) We conclude that $\gamma(A_0(M(-,r))$ is bounded independent of $r$ since there are only finitely many such curves.

Now, assume that $\varpi_1(A_0(M))$ is dense in $\C^2$ and  the closure of $\varpi_2(A_0(M))$ is a curve.  By Lemma~\ref{lemma:geoiso}, this occurs when the second cusp is geometrically isolated from the first cusp, but the first cusp is not geometrically isolated from the second cusp.  By Lemma~\ref{lemma:firstprojection}, 
 for any irreducible component $C$ of $A_0(M) \cap \mathcal{H}(r)$ with $A_0(M(-,r))$ contained in the closure of $\varpi_1(C)$,  
 \[\gamma (C)  \sim \gamma(A_0(M(-,r))).\]
Therefore, it suffices to bound $\gamma(C)$. By Lemma~\ref{lemma:key2},  $\gamma( C)\sim 1$.  (In this case the map $\varpi_2 \mid_{A_0(M)}$ does not have finite degree.)
\end{proof}

\begin{remark}

The proof of Theorem~\ref{theorem:1} implies that if the first cusp is geometrically isolated from the second, then $\varpi_1(A_0(M))$ is an irreducible curve and the curve  $A_0(M(-,r))$ is the same for all $r$ such that $\varpi_1(A_0(M))$  contains $A_0(M(-,r))$. A similar statement is true for $B_0(M(-,r))$. As there are only two characters of discrete faithful representations in $Y_0(M)$, there are at most two $B_0(M)$ and at most two projections.

It is possible that all $A_0(M(-,r))$ are the same if the second cusp is geometrically isolated from the first cusp.  This will occur if $\varpi_2(A_0(M) \cap \mathcal{H}(r))$ contains one of the common intersection points of the $\mathcal{C}(r)=\varpi_2(\mathcal{H}(r))$.  If not, we will show that the degree of $\varpi_1(A_0(M) \cap \mathcal{H}(r))$   is bounded by a constant because these curves are specializations of $A_0(M)$ at particular $(m_2,l_2)$ values. 
 
  \end{remark}

\begin{remark}   By Proposition~\ref{prop:relatedvarieties}  the bounds in Theorem~\ref{theorem:1} translate to bounds for other sets as well.  We obtain an analogous statement for  each of the sets
\[   A_0(M(-,r)),  B_0( M(-,r)), E_0(M(-,r)), X_0(M(-,r)), Y_0(M(-,r)). \]
For a given  $r$ such that $M(-,r)$ is hyperbolic, the correspondence between two such sets is established by a mapping between the sets.  Even though this is ambiguous, any lift can be traced back to $Y_0$ so there is a correspondence between different lifts.  
\end{remark}

\begin{remark}

 Assume that $M$ has two cusps neither of which is geometrically isolated from the other. (We follow  \cite{MR1426004} and use the notation established above.) With $r=p/q$, the core curve of the second cusp of $M(-,r)$ is given by $\gamma=\mu_2^{-b}\lambda_2^a$ with $ap+qb=1$.  We choose a parameter $T$ so that $T^{-q}=m_2$ and $T^p=l_2$, so that $m_2^pl_2^q=1$. With $\xi$ the (distinguished) eigenvalue of $\rho(\gamma)$, 
\[ \xi = m_2^{-b}l_2^a=T^{bq+ap} = T.\]
 The map  used to  bound  the gonality is 
\[ (m_1,l_1,m_2,l_2) \mapsto (m_2,l_2).\]
The image under this mapping has gonality equal to one as seen by the mapping 
\[ (m_2,l_2) \mapsto (m_2^bl_2^{-a},m_2^{p}l_2^{q}).\]
Using the  parametrization by $T$ this is
\[ (m_2,l_2)=(T^{-q},T^p) \mapsto (T^{-1}, 1).\]
After a final projection onto the first coordinate, since $T=\xi$, we see that the map is given by 
$ (m_1,l_1,m_2,l_2) \mapsto \xi^{-1}.$
Since sending a single variable to its inverse is a birational map, this is birationally equivalent to sending $ (m_1,l_1,m_2,l_2) $ to $\xi$.

\end{remark}

\section{Genus and Degree}\label{section:genusanddegree}

The gonality of a curve is intimately related to both the (geometric) genus of the curve, and the degree of the curve.  Like gonality, the genus is a birational invariant.  The degree of a curve depends on the embedding of the curve.   Assume that $C$ is a curve with a specified embedding into some ambient space.  Let $g$, $d$, and $\gamma$ be the genus, degree, and gonality of $C$, respectively. The Brill-Noether theorem relates the genus and the gonality of a curve. It states that 
\[ \gamma \leq \lfloor \frac{g+3}2 \rfloor.\]
The genus degree formula (the adjunction formula) states that  if $C$ is a smooth plane curve then
\[ g=\tfrac12(d-1)(d-2).\]
Each singularity of order $s$ reduces the right hand side by $\tfrac12s(s-1)$.

With Theorem~\ref{theorem:1} we can establish  bounds for the genus.  We have shown that the gonality is bounded. 
We now consider the variety $A_0(M(-,r))$ which naturally is contained in $\C^2(m,l)$.

\begin{reptheorem}{theorem:degree} {\em
Let $M$ be a finite volume hyperbolic 3-manifold with two cusps. If $M(-,r)$ is hyperbolic, then  there is a  positive  constant $c_1$ depending only on $M$ and the framing of the second cusp such that
\[ d(A_0(M(-,r))) \leq  c_1\cdot h(r).\] 
If one cusp is geometrically isolated from the other cusp, then there is a  positive  constant $c_2$ depending only on $M$ and the framing of the second cusp, such that 
\[ d(A_0(M(-,r))) \leq c_2.\]  }
\end{reptheorem}

\begin{proof}
First, we will bound the degree of (any irreducible component of) the intersection of a surface $X$ in $\C^4(x_1,y_1,x_2,y_2)$ with $\mathcal{H}(r)$.   As the dimension of $\mathcal{H}(r)$ is three and the dimension of $X$ is two, the condition that they intersect in a curve implies that the intersection is proper.  It is a consequence of B\'{e}zout's theorem (see \cite{MR0463157} Theorem I,7.7)  that with $r=p/q$, since $\deg \mathcal{H}(r) =\max\{|p|,|q|\}$, 
the degree of the intersection  is at most 
\[ \deg (X \cap \mathcal{H}(r) )\leq \deg X \cdot \deg \mathcal{H}(r)= \deg X  \cdot \max\{ |p|, |q| \}.\]

For any variety $W\subset \C^4$ the degree of $\varpi_1(W)$ is at most the degree of $W$.  To see this, assume that $W$ intersects a complimentary hyperplane $H$ in the point $P$.  Then $\varpi_1(W)$ intersects $\varpi_1(H)$ in $\varpi_1(P)$.  It suffices to see that we can take $H$ so that $\varpi_1(H)$ is a complimentary hyperplane to $\varpi_1(W)$ and that this respects multiplicities.  

Therefore, with $X=A_0(M)$ by Lemma~\ref{lemma:containment}, $A_0(M(-,r))$ is contained in the closure of $ \varpi_1(A_0(M) \cap \mathcal{H}(r))$.  
 We conclude that
 \[  \deg A_0(M(-,r)) \leq   \deg (\varpi_1(A_0(M) \cap \mathcal{H}(r)))\leq  \deg (A_0(M) \cap \mathcal{H}(r))\]
 and by  B\'{e}zout's theorem
 \[  \deg (A_0(M) \cap \mathcal{H}(r)) \leq  \deg A_0(M)  \max\{ |p|, |q|\}.\]
This implies the first assertion.

Now assume that the second cusp is geometrically isolated from the first cusp. Therefore $\varpi_2(A_0(M))$ is a curve. 
Let $\{ \varphi_i \}$ be a generating set for  $\mathfrak{U}_{A_0(M)}$ where $\varphi_i=\varphi_i(x_1,y_1,x_2,y_2)$.   Then $\mathfrak{U}_{A_0(M)\cap \mathcal{H}(r)} $ is generated by $\{\varphi_i(x_1,y_1,a_r,b_r)\}$ as in the proof of Lemma~\ref{lemma:key2}.  The degree of $A_0(M)\cap \mathcal{H}(r)$ is bounded by the degrees of  these equations. The map $\varpi_1$ is a projection.  It follows that degree of the the image $\varpi_1(A_0(M)\cap \mathcal{H}(r))$ is bounded by the degree (in $x_1$ and $y_1$) of the $\varpi_i(x_1,y_1,x_2,y_2)$.  

If the first cusp is geometrically isolated from the second cusp, the result follows from Lemma~\ref{lemma:geoisosamecurve}.
\end{proof}

We can bound the genus using the genus-degree formula. Specifically, we have the following.

\begin{reptheorem}{theorem:genus}{\em
Let $M$ be a finite volume hyperbolic 3-manifold with two cusps. If $M(-,r)$ is hyperbolic, there is a  positive  constant $c$  depending only on $M$ and the framing of the second cusp such that  \[   g(A_0(M(-,r))) \leq c \cdot h(r)^2.\]}
\end{reptheorem}

\begin{proof} It suffices to consider curves of degree more than two, since the genus of plane curves of smaller degree at most two is bounded. The theorem follows directly from the genus-degree formula which implies that if $C$ is a  plane curve of degree $d$ then the genus of $C$ is at most $\tfrac12(d-1)(d-2)$, which is less than $\tfrac12d^2$. By Theorem~\ref{theorem:degree} there is a constant $c$ such that  $d\leq ch(r)$  so that the genus us bounded by $\tfrac12c^2 h(r)^2$.

\end{proof}

\begin{remark}
For $r=p/q$ with $pq\neq 0$  this is 
\[ g(A_0(M(-,r))) \leq  c \max\{p^2, q^2 \}. \]
\end{remark}

This statement can also be made for $B_0( M(-,r))$.
  These genus bounds do not appear sharp.   In \S~\ref{section:doubletwist} we will show that  the double twist knots the genus is linear in the number of twists in each twist region. Similarly, the genus bounds from \cite{OPTBTNO} are linear in $h(r)$.

\section{Newton Polygons}\label{section:Newton}

Bounds for the genus and the gonality of a plane curve can be obtained from the associated Newton Polygon.  Let $f\in \C[x^{\pm 1}, y^{\pm 1}]$ be an irreducible Laurent polynomial which defines a curve $V(f)$ in $(\C-\{0\})^2$.  Let $\Delta(f)$ be the associated Newton polygon of $f$.  In 1893, Baker related the genus of such a curve to the volume of the Newton Polygon. 

\begin{thm}[Baker] The geometric genus of $V(f)$ is at most the number of lattice points in the interior of $\Delta(f)$. \end{thm}

Recently, Castryck and Cools showed a similar relationship for the gonality of $V(f)$.   A $\Z$-affine transformation is a map from $\R^2$ to $\R^2$ of the form $x\mapsto Ax+b$ where $A\in \GL_2(\Z)$ and $b\in \Z^2$.  We say that two lattice polygons $\Delta$ and $\Delta'$ are equivalent if there is a $\Z$-affine transformation $\varphi$ such that $\varphi(\Delta)=\Delta'$.  The lattice width of $\Delta$ is the smallest integer $s\geq 0$ such that there is a $\Z$-affine transformation $\varphi$ such that $\varphi(\Delta)$ is contained in the horizontal strip $\{ (x,y)\in \R^2 : 0\leq y \leq s \}$.

\begin{thm}[Castryck-Cools]
The gonality of $V(f)$ is at most the lattice width of $\Delta(f)$.
\end{thm}

Culler and Shalen \cite{MR735339} showed that  if $M$ has only one cusp, then $X_0(M)$ determines a norm on $H_1(\partial M, \R)$.  This norm encodes many topological properties of $M$ including information about the Dehn fillings of $M$.  This construction can be applied to other curve components of $X(M)$ yielding  a semi norm in general.   Culler and Shalen, and  Boyer and Zhang \cite{MR1670053}   used these semi norms to study surfaces in 3-manifolds.  Boyer and Zhang also extended this semi norm to $Y(M)$.

We observe the following about non-norm curve components.

\begin{thm}\label{thm:newtnorm}
 Let $M$ be a compact, connected, irreducible, $\partial$-irreducible 3-manifold with boundary  torus.  Let $Y\subset Y(M)$ be a curve which is not a norm curve.  Then  $Y$ is birational to $\C$. 
\end{thm}

\begin{proof}

By Boyer and Zhang \cite{MR1670053} (Proposition 5.4 3), if $M$ has one cusp any norm curve of $X(M)$ maps to a component of the $A$-polynomial whose Newton polygon is two-dimensional, and any curve that is not a norm curve maps to a component of the $A$-polynomial that is one-dimensional or zero-dimensional. Therefore, the Newton polygon of a non-norm curve is a line segment or a point.  This has empty interior, so by Baker's theorem, or Castryck and Cools' Theorem, the geometric genus is zero.  It follows that  the curve is birational to $\C$. 
\end{proof}

\section{Examples}

Theorem~\ref{theorem:1} shows that  the gonality of $X_0(M(-,r))$ is bounded independent of $r$.  In general, it is  difficult to  determine the exact gonality of such curves, as explicit computations of character varieties are difficult and only a few infinite families have been computed. In this section, we will compute the gonality of the character varieties associated to two families of manifolds which arise as Dehn surgeries.  These examples will demonstrate that the linear  bounds in Theorem~\ref{theorem:degree} are sharp, and moreover that there are character varieties with arbitrarily large gonality.  All of the filled manifolds in the following examples are knot complements in $S^3$.  For these $M(-,r)$,  $A_0(M(-,r))$ is birational to $X_0(M(-,r))$.

\subsection{Once-punctured torus bundles of tunnel number one}\label{section:torusbundles}

\begin{figure}[h!]
\begin{center}
\includegraphics[scale=.3]{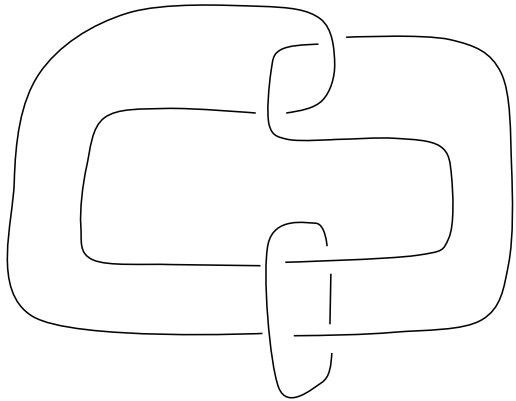}
\end{center}
\caption{The once-punctured torus bundle of tunnel number one, $M_n$,  is $-(n+2)$ surgery on one component of the Whitehead link, shown above.}
\label{fig:whitehead}
\end{figure}

The family of once-punctured torus bundles of tunnel number one is an infinite family of one-cusped manifolds, indexed as $\{M_n\}$ for $n\in \Z$.  These manifolds are hyperbolic for $|n|>2$ and can be realized as integral fillings of one component of the Whitehead link complement, $W$; the manifold $M_n$ is $W(-,-(n+2))$.  

The character varieties of these manifolds were explicitly computed in \cite{OPTBTNO}.   
For each hyperbolic $M_n$ there is a single canonical component in the $\SL_2(\C)$ character variety.   The curves $X_0(M_n)$ are all hyperelliptic curves of genus $\lfloor \tfrac12(|n-1|-1)  \rfloor $ if $n\neq 2 \pmod 4$ and $\lfloor \tfrac12(|n-1|-3)  \rfloor $ if $n\equiv 2 \pmod 4$.  (The variety $X_0(M_3)$ associated to the figure-eight knot sister is a  a rational curve. The variety $X_0(M_{-3})$ associated to the figure-eight knot complement and the variety $X_0(M_6)$ are elliptic curves.)   The discrepancy arising when $n\equiv 2 \pmod 4$ is due to the fact that in this case there are two components of irreducible representations, the canonical component and an affine line component. 

For  this family  $r=p/q=(-(n+2))/1$; the gonality of $X_0(M_n)$ is bounded by 2, and the genus is increasing linearly in $p$. 
Note that Theorem~\ref{theorem:genus} gives an upper bound of the form $c \cdot p^2$
for the genus. The canonical components of the $\PSL_2(\C)$ character variety are all birational to $\C^1$ and we conclude that here both the gonality and the genus are bounded by the constant one.

\subsection{Double Twist Knots}\label{section:doubletwist}

The  double twist knots are the knots $J(k,l)$ shown in the right hand side of  Figure~\ref{fig:threecomponents}.  These knots are the result of $-1/k$ surgery and $-1/l$ surgery on two components of the link shown in the left hand side of Figure~\ref{fig:threecomponents} where $k$ and $l$ are the number of half twists in the boxes.   The link $J(k,l)$ is a knot if $kl$ is even and otherwise is a two component link. Since $J(k,l)$ is ambient isotopic to $J(l,k)$ we may assume that $l=2n$.  The knot complement of $J(k,l)$ is hyperbolic unless $|k|$ or $|l|$ is less than two, or $k=l=\pm 2$.   The knots with $|k|$ or $|l|$ equal to two are the twist knots. These  include the figure-eight knot, $J(2,-2)=J(-2,2)=J(2,3)$ and the trefoil $J(2,2)=J(-2,-2)$. We now restrict our attention to these hyperbolic knots.

Let $Y_0(k,l)$ denote $Y_0(S^3-J(k,l))$, and $X_0(k,l)$ denote $X_0(S^3-J(k,l))$. For a fixed $k_0$,  by Theorem~\ref{theorem:1} 
\[ \gamma(Y_0(k_0,l))\sim c=c(k_0).\] 
A similar statement is true if we fix $l$.

\begin{figure}[h!]
\begin{center}
\includegraphics[scale=.5]{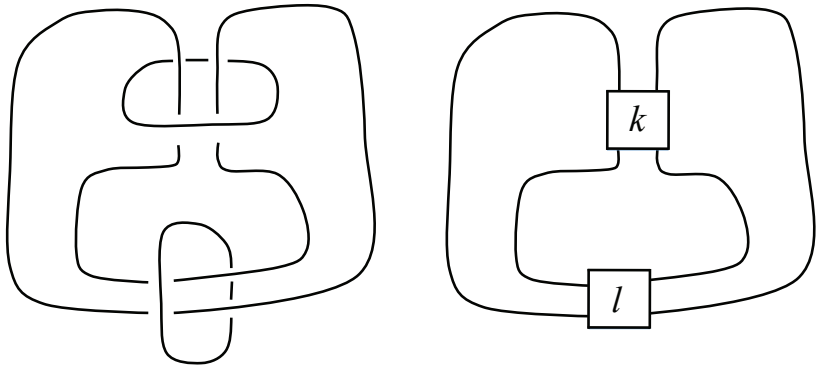}
\end{center}
\vspace{-5cm}
\caption{The link $J(k,l)$ is the result of $-1/k$ and $-1/l$ surgery on the three component link pictured on the left.}
\label{fig:threecomponents}
\end{figure}

The character varieties of these  knot complements  were computed in and analyzed in  \cite{MR2827003}.  The  component of the $\PSL_2(\C)$ character variety of the double twist knot $S^3-J(k,l)$  which contain irreducible representations is birationally equivalent to the curve defined by the following equations in $\PP^1 \times \PP^1$:
\begin{equation*}\tag{$*$}
\begin{aligned}
g_{m'+1}(r)g_{n'}(t) =g_{m'}(r) g_{n'+1}(t) & \ \text{ if } \ k=2m', l=2n'\\
f_{m'+1}(r) g_{n'}(t) = f_{m'}(r) g_{n'+1}(t)  & \ \text{ if } \  k=2m'+1, l=2n'. 
\end{aligned}
\end{equation*}
(The polynomials $f_i$ and $g_i$ were introduced in Definition~\ref{definition:fib}.) These curves are   smooth and irreducible in $\PP^1 \times \PP^1$ when $k\neq l$. When $k=l$ the canonical component is given by $r=t$.  The genus can be computed from the bidegree of these polynomials.
Let $m=\lfloor \tfrac{|k|}2 \rfloor$ and $n=\lfloor \tfrac{|l|}2 \rfloor$ and let $D(k,l)$ be the projective closure of the variety determined by the equations $(*)$ in $\PP^1(r) \times \PP^1(t)$.  
When $k\neq l$ the bidegree is  $(m,n)$.    The genus of $X_0(k,l)$ is determined by studying the ramification of the map from $X_0(k,l)$ to $Y_0(k,l)$. 

\begin{theorem}[\cite{MR2827003} Theorems 6.2 and 6.5]\label{theorem:mpvl}
Let $J(k,l)$ be a hyperbolic with $m=\lfloor \tfrac{|k|}2 \rfloor$.
If   $k\neq l$ the genus of $Y_0(k,l)$ is $(m - 1)(n  - 1)$ and the genus of $X_0(k,l)$ is $3mn-m-n- b$ where 
\[ b= \left\{ \begin{aligned} (-1)^{k+l} &  \ \  \text{  if   } \quad (-1)^{k+l}kl >0 \\ 0 & \quad \text{ otherwise} \end{aligned} \right.\]

The genus  of $Y_0(k,k)$ is zero and the genus of $X_0(k,k)$ is $n-1$. 
\end{theorem}

We will compute the gonality of these curves by relating the gonality to the degree of the curves.  
In 1883 Noether  \cite{noether} proved that a  non-singular plane curve of degree $d\geq 2$ in $\PP^2$ has gonality equal to  $d-1$, where $d$ is the degree of the curve.
(See \cite{MR555241} Theorem 2.3.1, page 73.)  Certain well-behaved singularities are known to modify this relationship in an understood manner.   The curves determined in \cite{MR2827003} are smooth, but lie in $\PP^1 \times \PP^1$; the birational mapping to $\PP^2$ introduces singularities in the curves.  We now prove a $\PP^1 \times \PP^1$ version of  Noether's theorem by understanding this mapping.

\begin{lemma}\label{lemma:bidegreegonality} 
Let $C$ be a smooth irreducible curve in $\PP^1 \times \PP^1$ of bidegree $(M,N)$ where $MN\neq 0$,   then  $\gamma(C)= \min\{M,N\}$. \end{lemma}

(If $M=0$ or $N=0$ then $C$ is isomorphic to $\PP^1$ and therefore has gonality one.)

\begin{proof}

Consider an irreducible curve $C\subset \PP^2$ of degree $d$.  Let $\nu$ be the maximum multiplicity of a singular point of $C$.  The projection map from the point with multiplicity $\nu$ to a line has degree $d-\nu$, so we conclude that $\gamma(C)\leq d-\nu$.   
Ohkouchi and  Sakai \cite{MR2060080} (Theorem 3) show that if $C$ has  at most two singular points (or infinitely near singular points), then $\gamma(C)=d-\nu$. 
(See \cite{MR1284715} page 194 for a discussion of infinitely near points.)

We consider the birational mapping $\phi:\PP^1 \times \PP^1 \rightarrow \PP^2$ given by  
\[ ([x:w],[y:z])\mapsto [wy:zx:xy].\]  This has birational inverse  given by 
\[ \phi^{-1}([a:b:c])=([c:a],[c:b]).\]  
Let $D$ be a smooth curve of bidegree $(M,N)$ in $\PP^1 \times \PP^1$
with $M\geq N$.  It suffices to show that $\nu = M$,
$deg(\phi(D))=M+N$, and $\phi(D)$ has exactly two singular or
infinitely near singular points; for then Ohkouchi and Sakai's result
implies that $\gamma(\phi(D))= d-\nu =N$. Since $\phi$ is birational,
$\gamma(D)=\gamma(\phi(D))$ proving the result.

 If $D$ is a smooth curve of bidegree $(M,N)$ then $D$ is the vanishing set of a polynomial $f(x,y)$ after projectivizing in $\PP^1 \times \PP^1$.  If $f(x,y) = \sum c_{n,m}x^ny^m$ the projectivization of $f$ is 
\[ F(x,w,y,z) = \sum c_{n,m} x^nw^{N-n} y^m z^{M-m}.\]
When $x=0$ (therefore $w=1$) this reduces to $\sum c_{0,M}y^Mz^{M-m}$. We conclude that there are $M$ points on $D$ such that $x=0$. Similarly there are $N$ points corresponding to  $y=0$ (and $z=1$).  When the polynomial has a constant term, then these points do not coincide.  If $f$ has no constant term, then $([0:1],[0:1]) \in D$.  We will assume, by taking an isomorphic copy of $D$, if necessary, that $([0:1],[0:1])\not \in D$,  so $f$ has a constant term.

Under the birational mapping,  $\phi([0:1],[y:z])=[y:0:0]=[1:0:0]$ (if $y\neq 0$).  Likewise, $\phi([x:w],[0:1])=[0:x:0]=[0:1:0]$ (if $x\neq 0$).  It follows that the $N$ points above get identified, as do the $M$ points. The image  $C=\phi(D)$  in $\PP^2 $ has a singularity of order $M$ and one of order $N$.  It is straightforward to see that if $xy\neq 0$ then a neighborhood of the point $([x:w],[y:z])$ maps smoothly into $\PP^2$. 
Therefore $\phi(D)$ has exactly two singular points, one of degree $M$ and one of degree $N$.

The degree of $C$ is given by the intersection number $CH$ where $H'$ is a hyperplane section of $\PP^2$.  Similarly, the degree of $D$ is the intersection number $DH$ where $H$ is a hyperplane section of $\PP^1 \times \PP^1$.  The divisor class group of $\PP^1 \times \PP^1$ is  $Cl (\PP^1 \times \PP^1)=\Z \oplus \Z$. Let $E=\PP^1\times\{x\}$ and $F=\{x\} \times \PP^1$.   Then  $D$ is linearly equivalent to $ME+NF$, that is $D\sim ME+NF$. (Two divisors are linear equivalent if their difference is a principal divisor.) The curves of bidegree $(M,N)$ correspond to the effective divisors of the class $ME+NF$.    Similarly, $H\sim E+F$. 

Therefore, since the degree of a divisor depends only on the linear equivalence class, $deg(D)=deg(ME+NF)$.  
Since the degree of $D$ is the intersection number $DH$, and $D\sim ME+NF$ and $H\sim E+F$ it follows that the degree of $D$ is the intersection number of $ME+NF$ and $E+F$, which is  $M+N$.

\end{proof}

We now use this to determine the gonality of the curves $Y_0(k,l)$.

\begin{thm}\label{thm:doubletwisty}
Let $J(k,l)$ be a hyperbolic knot  and let $m=\lfloor \tfrac12|k| \rfloor $ and $n=\lfloor \tfrac12 |l| \rfloor $.
\begin{itemize}
\item[] If $k\neq l$ then   $\gamma(Y_0(k,l))=\min\{ m,n\}$.  
\item[]If $k=l$ then    $\gamma(Y_0(k,l))=1$.
\end{itemize}
\end{thm}

\begin{proof}
The equations $(*)$ defining $Y_0(k,l)$ are irreducible and determine smooth curves in $\PP^1 \times \PP^1$ by 
 \cite{MR2827003}.  These equations have bidegree $(m,n)$ if $k\neq l$.  By Lemma~\ref{lemma:bidegreegonality}  if $nm\neq 0$ the gonality of $Y_0(k,l)$ is $\min\{m,n\}$.  
If $m,n=0$ then $k=0,\pm1$ or $l=0,\pm 1$ and $J(k,l)$ is not hyperbolic. It suffices to consider the case when $k=l=2n'$.   In this case, the canonical component for the $\PSL_2(\C)$ character variety is isomorphic to $\PP^1 \times \PP^1$, given by $r=t$. We conclude that the gonality is one.

\end{proof}

\begin{cor}
For all positive integers $n$ there are infinitely many  double twist knots  $K$ such that \[ \gamma(Y_0(S^3-K))=n.\]
\end{cor}

Now we compute the gonality of the canonical components of the  $\SL_2(\C)$ character varieties.  
By Lemma~\ref{lemma:abstractgonalitybounds} since $X_0(k,l)$ is a degree 2 branched cover of $Y_0(k,l)$ we conclude that 
\[ \min\{ m,n\} \leq \gamma(X_0(k,l)) \leq 2  \min\{ m,n\}.\] 
We now explicitly show that the gonality equals this upper bound.

\begin{thm}\label{thm:doubletwisty2}
Let $J(k,l)$ be a  hyperbolic knot and let $m=\lfloor \tfrac12|k| \rfloor $ and $n=\lfloor \tfrac12 |l| \rfloor $.
\begin{itemize}
\item[] If $k\neq l$ then $\gamma(X_0(k,l))=2 \min\{ m,n\}$.  
\item[]If $k=l$ then  $\gamma(X_0(k,l))=2$. 
\end{itemize}
\end{thm}

\begin{proof} 

We first  consider the case when $k=l=2n'$. By  \cite{MR2827003}, a  `natural model' for  $X_0(2n',2n')\subset \C^2(x,r)$ is given by 
\[ x^2 = \frac{(r+2)f_{n'}(r)^2-1}{f_{n'}(r)^2}\]
and is birational to the hyper elliptic curve
\[ w^2=(r+2)f_{n'}(r)^2-1.\]
This has genus $n'-1$ by \cite{OPTBTNO}.
If $J(k,k)$ is hyperbolic then  $|n'|>2$ and the gonality of $X_0(k,k)$ is two.

Now,  assume that $k\neq l$ and $J(k,l)$ is hyperbolic.  If $m$ or $n$ is one, then one can verify directly from the defining equations that $X_0(k,l)$ is an elliptic or  hyperelliptic curve and therefore has gonality two.  We will now assume that $m$ and $n$ are at least two.  

Consider the following situation.   Let $X$ and $Y$ be smooth curves of genus $g_X$ and $g_Y$ and let $f:X\rightarrow Y$ be a degree $k$ simple morphism 
and $h:X\rightarrow \PP^1$ a degree $d$ morphism. (A simple morphism is one that does not factor through a nontrivial morphism.)  The Castelnuovo-Severi inequality (see \cite{MR2846309} Theorem 2.1) states that if 
\begin{equation*}\tag{$*$} d \leq \frac{g_X-k g_Y}{k-1} \end{equation*}
then $h$ factors through $f$.

Let $X$ be a smooth projective model for $X_0(k,l)$ and $Y$ a smooth projective model for $Y_0(k,l)$. Since there is a degree two birational map from $X_0(k,l)$ to $Y_0(k,l)$, by Stein factorization, there is a degree two morphism from $X$ to $Y$. 


From \cite{MR2827003},  we have a degree $k=2$ mapping, with $g_X=3mn-m-n-b $  (for an explicitly determined $b\in \{0,\pm1\}$) and $g_Y=(m-1)(n-1)$.  Therefore, the righthand side of the inequality is 
\[ g_X-2g_Y=  (3mn-m-n-b) -2(m-1)(n-1) =mn+m+n-b-2.     \]
Since $b\in \{0,\pm 1\}$ the right-hand side of $(*)$ is at most $mn+m+n-1$. 
We conclude that any dense map from $X_0(k,l)$ to $\PP^1$ of degree at most $mn+m+n-1$  factors through a map from $Y_0(k,l)$ to $\PP^1$.     Let $d$ be the minimal degree of a map from $X_0(k,l)$ to $\PP^1$.  Since the gonality of $Y_0(k,l)$ is $\min\{m,n\}$ we conclude that 
\[ \min \{m,n\}\leq d \leq 2 \min \{ m,n\}.\]  
Since $m$ and $n$ are at least two,  
\[ 
d \leq 2 \min\{m,n\} \leq mn+m+n-1.
\]  
By the Castelnuovo-Severi inequality,  the map factors through $Y_0(k,l)$.  As the gonality of $Y_0(k,l)$ is $\min\{m,n\}$ by Lemma~\ref{thm:doubletwisty} we conclude that the smallest degree of a map from $X_0(k,l) \rightarrow \PP^1$ is $2\min\{m,n\}$.

\end{proof}

\begin{remark} The Brill-Noether theorem gives the bound $\gamma \leq \lfloor \frac{g+3}{2} \rfloor$ where $g$ is the genus. For $Y_0(k,l)$ this inequality is
\[ 
\min\{m,n\} = \gamma \leq \lfloor   \tfrac12(mn-m-n+2)   \rfloor
\]
so the gonality is smaller than the  Brill-Noether bound.
 For $k=l$ the genus of $Y_0(k,l)$ is zero and the gonality is one, meeting the Brill-Noether bound.  The examples with $k\neq l$ show that for any $N$ there is a double twist knot (a twist knot, in fact) such that $\tfrac12(g+3)-\gamma >N$.   That is, the discrepancy between the gonality and the Brill-Noether bound is unbounded  in this family. 
 \end{remark}

\begin{remark}
Hoste and Shanahan have computed a recursive formula for the $A$-polynomials of the twist knots \cite{MR2047468}, the knots $J(\pm2, n)$ and shown that these polynomials are irreducible. 
Their results show that the degree of the $A$-polynomials 
grow linearly in $n$.  This shows that the linear bounds in Theorem~\ref{theorem:degree} 
are sharp, since the twist knots are $-1/n$ surgery on one component of the Whitehead link. 
\end{remark}

\section{Gonality, Injectivity Radius, and Eigenvalues of the Laplacian}


Since the smooth projective model of a (possibly singular) affine curve is
a compact Riemann surface, it can also be viewed as
a quotient of the Poincare disk by a discrete group
of isometries of the Poincare disk, and as such it is natural to try to relate gonality to geometric and analytical properties associated to the hyperbolic
metric; for example 
the injectivity radius, $\varrho$ (i.e. half of the length of 
a shortest closed geodesic),  
or the first non-zero eigenvalue of the Laplacian,  $\lambda_1$.

To that end, a famous result of Li and Yau \cite{MR674407} shows that
for a compact Riemann surface of genus $g$, a bound for the gonality 
$\gamma$ is given by:

\[ \lambda_1 (g-1) \leq \gamma\]
where $\gamma$ is the gonality defined using regular maps. Since the (rational) gonality of a complex affine algebraic curve equals the (regular) gonality of a smooth projective completion, this bound applies to smooth projective models of  character varieties.  A Riemann surface of genus zero is rational 
and so has gonality one, and a Riemann surface of genus one is an elliptic curve and so has gonality two.  For any higher genus, the above bound implies that 
\[ \lambda_1 \leq \frac{\gamma}{g-1}.\]
 By Theorem~\ref{theorem:1}  the (rational) gonality of $X_0(M(-,r))$ is bounded independent of $r$, and therefore so is the (regular) gonality, $\gamma$, of a smooth projective model. We conclude that $\lambda_1$ is bounded above independent of $r$ for these curves.  
Moreover, if $\{r\}$ is sequence of rational numbers  such that  the genus of the character varieties  grows, then since $\gamma$ is bounded 
\[ \lambda_1  \leq \frac{\gamma}{g-1} \rightarrow 0.\]

For hyperbolic double twist knots the genera of $Y_0(k,l)$ and  $X_0(k,l)$ was determined in \cite{MR2827003}. (See \S~\ref{section:doubletwist}.)   The genus of $X_0(k,l)$ is less than two only for the complements of the knots  $J(2,-2)$ (the figure-eight) and $J(4,4)$; both  have genus 1.  The only hyperbolic double twist knots such that the genus of $Y_0(k,l)$ is less than two are  $J(k,k)$, $J(2,-2)$, $J(2,-3)$ which have genus zero, and $J(4, \pm 5)$ which have genus one. 
When $k\neq l$, we obtain the following inequalities using Theorem~\ref{theorem:mpvl}
 \[ \lambda_1(Y_0(k,l))  \leq \frac{ \min\{m,n\}}{(m-1)(n-1)} , \qquad \lambda_1(X_0(k,l)) \leq \frac{2\min\{m,n\}}{3mn-m-n-b}  \]
 where $X_0(k,l)$ and $Y_0(k,l)$ denote smooth projective models and  $n= \lfloor \tfrac12|k| \rfloor  $ and $m = \lfloor \tfrac12|l| \rfloor.  $
Both quantities $\rightarrow 0$  as $m$ or $n$  $\rightarrow \infty.$ 
For the once-punctured torus bundles of tunnel number one, (see \S~\ref{section:torusbundles}) let $X_0(M_n)$ denote a smooth projective model.
Then for  $|n|>4$ 
 \[ \lambda_1(X_0(M_n))  \leq \frac{2}{\lfloor  \tfrac12(|n-2|-2) \rfloor} \leq \frac4{|n-1|-3} \]
and we conclude that  $\lambda_1 \rightarrow 0$ as $|n| \rightarrow \infty$.

More recently Hwang and To \cite{MR2876146}  show that the injectivity radius $\varrho$ of a compact Riemann surface of genus $g\geq 2$ is bounded above by a function of gonality (defined for regular maps),
\[ \varrho \leq 2 \cosh^{-1} (\gamma). \]
With $X_0(M(-,r))$ denoting a smooth projective model, 
Theorem~\ref{theorem:1} implies that there is a constant $c$ such that the gonality of these varieties is bounded above by $c$ and therefore there is a constant $d$ such that 
\[ \varrho (X_0(M(-,r))) \leq 2 \cosh^{-1} (\gamma) =d. \]
It follows that  the injectivity radius is bounded above for all $X_0(M(-,r))$, such that  $g\neq 0,1$.

By Theorem~\ref{thm:doubletwisty2} and Theorem~\ref{thm:doubletwisty} for all  hyperbolic double twist knots other than the figure-eight, if $k\neq l$
\[ \varrho (X_0(k,l)) \leq 2 \cosh^{-1} (2 \min \{m,n\}), \quad \text{and } \quad \varrho (Y_0(k,l)) \leq 2 \cosh^{-1} (\min \{m,n\}). \]
If $k=l$ then $\varrho (X_0(k,l)) \leq 2 \cosh^{-1}(2)$. 
For all  once-punctured torus bundles of tunnel number one with $|n|>6$,  $X_0(M_n)$ is a hyper elliptic curve. We conclude that for the corresponding smooth projective models  the injectivity radius is bounded by $2 \cosh^{-1} (2).$ 

Motivated by the discussion here, we close with two questions.  To state
these define an infinite family of compact Riemann surfaces $\{\Sigma_j\}$
to be {\em an expander family} if the genera of $\Sigma_j\rightarrow \infty$
and there is a constant $C>0$ for which $\lambda_1(\Sigma_j)>C$.  One such
family are congruence arithmetic surfaces.\\[\baselineskip]
\noindent{\bf Question:}~{\em (1)~Does there exist an infinite family 
of 1-cusped hyperbolic 3-manifolds $M_j$ for which $X_0(M_j)$ (resp. $Y_0(M_j)$)
is an expander family?}

\medskip

\noindent{\em (2) Does there exist an infinite family 
of 1-cusped hyperbolic 3-manifolds $M_j$ for which $X_0(M_j)$ (resp. $Y_0(M_j)$)
has arbitrarily large injectivity radius?}

\bibliographystyle{amsplain}
\bibliography{gonalityfinal.bib}

\end{document}